\documentclass[reqno,a4paper]{amsart} 
\usepackage{amssymb,mathtools,graphicx,ifthen,comment}
\usepackage[all]{xy} \xyoption{import}
\def\useshowkeys{0}
\ifthenelse{\equal{\useshowkeys}{1}}{\usepackage{showkeys}}{}
\usepackage{hyperref,psfrag,rotating,wrapfig}
\mathtoolsset{showonlyrefs}
\def\octave{{\sc Octave\ }}
\def\maxima{{\sc Maxima\ }}
\def\metric#1{{\bf #1}}
\def\y{{\bf y}}
\def\brisk#1{{\mathcal R}\ifthenelse{\equal{#1}{}}{}{(#1)}}
\def\ev{{\rm ev}}
\def\argmin{{\rm argmin\, }}
\def \implies{{\Longrightarrow}}

\def \dim{{\mbox {dim}}\,}

\def\im#1{{\rm Im\, }#1}

\def\ssqerr#1{{\mathsf s}\ifthenelse{\equal{#1}{}}{}{(#1)}}
\def\g{{\mathfrak g}}

\def\h{{\mathfrak h}}
\def\k{{\mathfrak k}}

\def\orb#1#2{#1 \hspace{-2pt}\cdot\hspace{-2pt} #2}
\def\stab#1#2{\ifthenelse{\equal{#1}{}}{{\rm stab}(#2)}{{#1}_{#2}}}

\def\R{{\bf R}}

\def \C{{\mathbf C}}
\def \H{{\mathbf H}}

\newcommand{\gl}[1] {{\rm gl}(#1)}

\def\dist {{\bf dist}}
\def\E{{{\bf E}}}
\def\d{{\rm d}}
\def\dvol{{\d \nu}}
\def\gauss{{\psi}}
\def\gb{{\tilde{g}}}

\def\briskk#1{\ifthenelse{\equal{#1}{}}{\tilde{R}_{\epsilon}}{\tilde{R}_{\epsilon}(#1)}}

\def\scal{{\rm scal}}
\def\cov{{\mathfrak c}}
\def\tube{{\rm T}}
\def\sym#1{{\rm sym}(#1)}
\def\pl{\Pi}
\def\diag#1{{\rm diag}(#1)}
\def\ds{\displaystyle}
\def\ie{{\it i.e.\ }}
\def\cf{{\it c.f.\ }}
\def\spn#1{{\rm span\,}#1}
\def\ddt#1{{\frac{\d\phantom{#1}}{\d #1}}}
\def\didi#1#2{{\frac{\partial\phantom{#1}}{\partial #1}#2}}
\def\cutlocus{{\mathfrak C}}
\def\gcutlocus{{\mathfrak C}_{\Gamma}}

\newcommand{\hzero}[1] {{\Huge{\bf 0}}\ifthenelse{\equal{#1}{}}{}{${}_{#1}$}}
\newcommand{\normal}[2] { {\mathcal N}(#1,#2) }
\newcommand{\p} [0] { {\bf p} }

\newcommand{\GL} [1] { {\rm{GL(}}#1 {\rm )} }

\newcommand{\SO} [1] {{\rm{SO}}\ifthenelse{\equal{#1}{}}{}{(#1)}}
\newcommand{\Orth} [1] {{\rm{O}}\ifthenelse{\equal{#1}{}}{}{(#1)}}
\newcommand{\orth} [1] {{\rm{o}}\ifthenelse{\equal{#1}{}}{}{(#1)}}

\newcommand{\so} [1] {{\rm{so}}\ifthenelse{\equal{#1}{}}{}{(#1)}}
\newcommand{\SU} [1] {{\rm{SU}}\ifthenelse{\equal{#1}{}}{}{(#1)}}
\newcommand{\su} [1] {{\rm{su}}\ifthenelse{\equal{#1}{}}{}{(#1)}}
\newcommand{\U} [1] {{\rm{U}}\ifthenelse{\equal{#1}{}}{}{(#1)}}
\newcommand{\un} [1] {{\rm{u}}\ifthenelse{\equal{#1}{}}{}{(#1)}}
\newcommand{\Sp} [1] {{\rm{Sp}}\ifthenelse{\equal{#1}{}}{}{(#1)}}
\newcommand{\symp} [1] {{\rm{sp}}\ifthenelse{\equal{#1}{}}{}{(#1)}}

\newcommand{\nb}[2][ ] {{\mathrm N}_{#1}(#2) }

\newcommand{\ad} [1] { {\rm ad}_{#1} }
\newcommand{\Ad} [1] { {\rm Ad}_{#1} }

\newcommand\Hom [1] {{\rm Hom(}#1{\rm )}}
\newcommand\stf [2] {V_{#1}(#2)}
\newcommand\gr [2] {G_{#1}(#2)}
\newcommand\Ric [1] {{\rm Ric}_{#1}}
\newcommand\trace [1] {{\rm\, Tr} \left(#1\right)}
\newcommand\vtheta [0] {{\vartheta}}
\newcommand\vTheta [0] {{\Theta}}

\newtheorem{theorem}{Theorem}[section]
\newtheorem{thm}{Theorem}[section]

\newtheorem{remark}[thm]{Remark}
\newtheorem{lemma}[thm]{Lemma}
\newtheorem{proposition}[thm]{Proposition}
\newtheorem{definition}[thm]{Definition}
\newtheorem{corollary}[thm]{Corollary}

\newtheorem{example}[thm]{Example}

\excludecomment{maximacode}
\excludecomment{arxivabstract}

\begin{document}

\title[Bayesian estimation of maps]{A Bayesian approach to the
  Estimation of maps between riemannian manifolds, II: Examples}

\author[Butler, Levit]{Leo T. Butler and Boris Levit}
\address{LB: School of Mathematics, 6214 James Clerk Maxwell Building,
  The University of Edinburgh,  Edinburgh, UK, EH9 3JZ
BL: Department of Mathematics and Statistics
Queen's University,
Kingston, ON, Canada, K7L 3N6}
\email{l.butler@ed.ac.uk, blevit@mast.queensu.ca}

\date{\today}

\thanks{The first author thanks the Carnegie Trust for the
  Universities of Scotland for supporting this research.}

\begin{abstract}
Let $\Theta$ be a smooth compact oriented manifold without boundary,
imbedded in a euclidean space ${\bf E}^s,$ and let $\gamma$ be a
smooth map of $\Theta$ into a Riemannian manifold $\Lambda$. An
unknown state $\theta\in \Theta$ is observed via $X=\theta+\epsilon
\xi$ where $\epsilon>0$ is a small parameter and $\xi$ is a white
Gaussian noise. For a given smooth prior $\lambda$ on $\Theta$ and
smooth estimators $g(X)$ of the map $\gamma$ we have derived a second-order
asymptotic expansion for the related Bayesian risk \cite{BL}. In this
paper, we apply this technique to a variety of examples.

The second part examines the first-order conditions for
equality-constrained regression problems. The geometric tools that are
utilised in \cite{BL} are naturally applicable to these regression
problems.
\end{abstract}
\maketitle

\begin{arxivabstract}
Let M be a smooth compact oriented manifold without boundary,
imbedded in a euclidean space E and let f be a smooth map of M
into a Riemannian manifold N. An unknown state x in M is observed
via X=x+su where s>0 is a small parameter and u is a white
Gaussian noise. For a given smooth prior on M and smooth
estimators g of the map f we have derived a second-order
asymptotic expansion for the related Bayesian risk (see
http://uk.arxiv.org/abs/0705.2540). In this paper, we apply this
technique to a variety of examples.

The second part examines the first-order conditions for
equality-constrained regression problems. The geometric tools
that are utilised in our earlier paper are naturally applicable
to these regression problems.
\end{arxivabstract}

\section{Introduction} 
\label{intro}

In many estimation problems, one has a state which lies on a manifold
but one observes this state plus some error in a euclidean space. It
is desirable to utilise the underlying geometry to construct an
estimator of the state. The present paper uses a Bayesian approach and
the Bayesian estimator derived in \cite{BL} and computes the estimator
in a variety of examples.

In many cases, the geometric framework of \cite{BL} naturally extends
to regression problems. In an estimation problem, the map is known
while the state is observed with noise and one attempts to infer the
`true' state; in a regression problem, the map is unknown and one
observes the input-output states with some noise and attempts to infer
the map. In this paper, we will assume that the regression map belongs
to a given compact finite-dimensional manifold. In such a situation,
one may formally transpose the regression problem in the sense that
one may regard the map as the state that one observes with noise and
the input-output states may be regarded as (evaluation) maps. This
transposition is commonly used in topology and differential
geometry. In the second part of this note, we derive first-order
conditions for regression problems on manifolds. It is shown in
several cases that this duality between estimation and regression is
exact: the two viewpoints lead to the same estimator.

Consider the following situation: $\E$ is a real $s$-dimensional
vector space with inner product $\sigma$ and $\Theta$
(resp. $\Lambda$) is a smooth manifold with riemannian metric ${\bf
g}$ (resp. $\metric{h}$). Assume that the smooth riemannian manifold
$(\Theta,\metric{g})$ is isometrically embedded in a euclidean space
$(\E,\sigma)$ via the inclusion map $\iota$, and $\Theta
\stackrel{\gamma}{\longrightarrow} \Lambda$ is a smooth map. Smooth
means infinitely differentiable. These data are summarized by the
diagram
\begin{equation} \label{eq:cd}
\xymatrix@!R@R=6pt{
\nb[]{\Theta} \ar[rr]^{\text{incl.}} \ar@/_3mm/[rrdd]_{\pi} && (\E,\sigma) \ar@{.>}[ddrr]^{g}\\
\\
&& \ar[uull]_{\text{incl.}} (\Theta,\metric{g}) \ar@{->}[uu]^{\iota} \ar[rr]^{\gamma} && (\Lambda,\metric{h}),
}
\end{equation}
where $\nb[]{\Theta}$ is an open neighbourhood of $\Theta$ in $\E$ and
$\pi$ is the orthogonal projection onto $\Theta$. A basic result of
differential geometry is that if $\Theta$ is compact, then there is an
$r>0$ such that $\pi$ is a smooth map on the set of all vectors
within a distance $r$ of $\Theta$ \cite{MT}.

Suppose that $X \in \E$ is a gaussian random variable with mean
$\theta \in \Theta$ and covariance operator
\footnote{By convention, the covariance operator is the induced inner
  product on the dual vector space $\E^*$. If we regard $\sigma$ as a
  linear isomorphism of $\E \to \E^*$, then the covariance operator is
  the inverse linear isomorphism $\cov=\sigma^{-1} : \E^* \to \E$. It
  is common to think of $\E$ as a space of column vectors, and the
  dual as a space of row vectors, in which case $\cov$ is the
  transpose map $x \mapsto x'$ from row to column
  vectors.}
\ $\epsilon^2 \cov$, \ie $$X \sim \normal{\theta}{\epsilon^2
  \cov},\ \ \ \ \theta \in \Theta.$$ A basic statistical problem is to
determine an estimator ``$\gamma(X)$,'' by which we mean an optimal
extension of $\gamma$ off $\Theta$, in the minimax sense. To make this
precise, let $g : \E \to \Lambda$ be an estimator (map), and let
$\dist$ be the riemannian distance function of
$(\Lambda,\metric{h})$. Define a loss function by
\begin{equation} \label{eq:qr}
R_{\epsilon}(g,\theta) = \int_{x \in \E}\,
\dist(g(x),\gamma(\theta))^2\,
\gauss_{\epsilon}(x-\iota(\theta))\, \d x,
\end{equation}
where $\gauss_{\epsilon}(u) = \exp(-|u|^2/2\epsilon^2 ) / (2\pi
\epsilon^2)^{\frac{s}{2}}$, $|\bullet|$ is the norm on $\E$ induced by
$\sigma$, and $\d x$ is the volume form on $\E$ induced by $\sigma$.%
\footnote{One can introduce a $\sigma$-orthonormal coordinate system
  $x_i$ on $\E$. In this case, $|x|^2 = \sum_i x_i^2$ and $\d x = \d
  x_1 \wedge \cdots \wedge \d x_s$. }  Define the associated minimax
  risk
\begin{equation} \label{eq:mmrisk}
r_{\epsilon}(\Theta) = \inf_{g}\, \sup_{\theta \in \Theta}\
R_{\epsilon}(g,\theta).
\end{equation}

\subsection{Results: Bayesian estimation} \label{ssec:results-estimation}
One may use a Bayesian approach to determine the asymptotically
minimax estimator $g$.  Here one views $\theta$ is viewed as a random
variable with a prior distribution $\lambda(\theta) \d\theta$ where
$\int_{\theta \in \Theta} \lambda(\theta)\, \d\theta = 1$ ($\d\theta =
\dvol_{{\bf g}}$ is the riemannian volume of $(\Theta,\metric{g})$
). The {\em Bayesian risk} of a map $g$ is
\begin{equation} \label{eq:brisk}
R_{\epsilon}(g;\lambda) = \displaystyle  \int_{\theta \in \Theta} \int_{x\in\E}
\dist(g(x),\gamma(\theta))^2\, \lambda(\theta)\,
\gauss_{\epsilon}(x-\iota(\theta))\ \d x\, \d\theta.
\end{equation}
A {\em Bayes estimator} $g : \E \to \Lambda$ is a map which minimizes
the Bayesian risk over all maps. 

Before stating the main result of \cite{BL}, recall that a riemannian
connection permits one to define higher-order derivatives. In
particular, $\nabla \d$ is used to denote the hessian (second
derivative) and $\tau = \trace{\nabla \d}$ denotes the tension field
(laplacian), while $\Ric{}$ denotes the Ricci curvature
\cite{EL:1971,BL}.

In \cite{BL}, the present authors proved

\begin{theorem} \label{thm:best}
Let $\gb_{\epsilon}(x) = \exp_{g_o(x)}\left( \epsilon^2 g_2(x) +
O(\epsilon^4)\right)$ be the Bayesian estimator for the Bayesian risk
functional $R_{\epsilon}$ (Equation \ref{eq:brisk}) with a fixed
Bayesian prior $\lambda>0$, where $g_o,g_2$ are the lowest order terms
in the expansion. Then for all $\epsilon$ sufficiently small
\begin{enumerate}
\item for all $x \in \nb[]{\Theta}$ with $|x-\pi(x)| \leq r$
$$\gb_{\epsilon}(x) = \exp_{\Gamma(x)}\left( \epsilon^2 \left[
  \frac{1}{2} \tau(\gamma) + \d\gamma(\nabla \log\lambda)
  \right]_{\pi(x)} + O(\epsilon^4) \right),$$\\ where $\Gamma
  =\gamma\pi$, and $\exp$ is the exponential map of
  $(\Lambda,\metric{h})$; and
\item
\begin{align*}
R_{\epsilon}(\gb_{\epsilon} ; \lambda) &= \epsilon^2\, \int \d\theta\,
\lambda\, |\d\gamma|^2 +\\ &\ \ \ \ \epsilon^4 \int \d\theta\,
\lambda\, \left\{ \frac{1}{2}| \nabla \d\Gamma |^2 -
|\tau(\gamma)+\d\gamma(\nabla \log \lambda)|^2 - \frac{2}{3} \langle
\d\Gamma, \Ric{\d\Gamma} \rangle \right\} + O(\epsilon^6).
\end{align*}

\end{enumerate}
\end{theorem}

In section~\ref{sec:homog-id}, this notes applies Theorem
\ref{thm:best} to compute the bayesian estimator and risk of the
identity map for a wide class of compact group orbits and a `linear'
prior (see Theorem \ref{thm:linear-priors}).

\subsection{Results: Bayesian regression } \label{ssec:results-regression}
Let $\Theta,\Lambda$ be smooth manifolds. Let
$\theta_1,\ldots,\theta_k$ be a collection of design points on a
manifold $\Theta$ and let $y_1,\ldots,y_k$ be a random sample of
points on $\Lambda$. Assume that the conditional probability density
of $y_l$ given $\theta_l$ is $f(y_l|\gamma(\theta_l))$, where $\gamma
: \Theta \to \Lambda$ is an unknown map and $\theta$ is a point on
$\Theta$.  One is interested in estimating the unknown map $\gamma$ by
minimising a discrepancy function
\begin{align}
\ssqerr{\gamma} &= \frac{1}{k} \ds\sum_{l=1}^k
\ell(y_l,\gamma(\theta_l))  \label{al:reg-sq-err-template}
\end{align}
for a given loss function $\ell : \Lambda \times \Lambda \to
[0,\infty)$ (see section \ref{ssec:reg-bayes}). If one assumes that
  the space of admissible maps $\gamma$ is parameterized by a
  finite-dimensional manifold $\Gamma$, a solution to this regression
  problem is
\begin{equation} \label{eq:reg-soln-template}
\hat{\gamma} = \argmin \left\{ \ssqerr{\gamma} \ :\ \gamma\in \Gamma\right\}.
\end{equation}

One may also consider the regression problem from a bayesian
perspective. In this case, one assumes there is a smooth volume form
$\d \gamma$ on $\Gamma$ and a prior distribution $\lambda(\gamma)\, \d
\gamma$. The bayesian regression problem is to derive the regressor by
minimising the risk functional
\begin{align} \label{al:bay-risk-0}
\brisk{\hat\gamma} &= \int_{\y \in \Lambda^k} \int_{\gamma\in \Gamma}
\ell(\hat \gamma, \gamma)\, f(\y|\gamma)\, \lambda(\gamma)\, \d\y\, \d \gamma.
\end{align}
over regressors $\hat\gamma : \Lambda^k \to \Gamma$.

In section \ref{sec:reg}, these two regression problems are
examined. First-order conditions that determine the regressors are
proven. In addition, we examine the special cases where
$\Theta,\Lambda \subset \E$, $\Gamma$ is a manifold of linear maps and
\begin{enumerate}
\item $\ell$ is determined by the ambient euclidean structure;
\item $\ell$ is determined by the riemannian distance function on
  $\Lambda$ induced by the euclidean structure.
\end{enumerate}
The special case where $\Theta$ and $\Lambda$ are both the
$2$-dimensional unit sphere $S^2 \subset \E^3$ and $\Gamma$ is the
group of orientation-preserving linear isometries of $\E^3$, $\SO{3}$
is examined in detail in each case.

The results of section \ref{sec:reg} are formulated in Propositions
\ref{pr:reg-foc}, \ref{pr:reg-id-foc} and \ref{pr:bay-foc}.

\section{Estimation of states on group orbits} \label{sec:homog-id}
Let $(\E^s,\sigma)$ be an $s$-dimensional euclidean space: that is,
$\E^s$ is an $s$-dimensional real vector space and $\sigma$ is a
symmetric, positive-definite quadratic form on $\E^s$. The group of
linear isometries of $\E^s$ is denoted by $\Orth{\E^s,\sigma}$ and
called the {\em orthogonal group} of $(\E^s,\sigma)$. This group is
denoted by $\Orth{\E}$ when the euclidean structure $\sigma$ and
dimension $s$ are understood. By choice of an orthonormal basis,
$(\E^s,\sigma)$ is linearly isometric to $\R^s$ with its standard
orthonormal basis; the orthogonal group of this latter model euclidean
space is denoted by $\Orth{}_s$, while $\SO{}_s$ is the subgroup of
$\Orth{}_s$ with unit determinant. 

The set of linear transformations $\E^s \to \E^r$ is denoted by
$\Hom{\E^s,\E^r}$. It is naturally a euclidean space with the trace
inner product $(x,y) \mapsto \trace{x'y}$. There is an orthogonal
decomposition of $\Hom{\E^s,\E^s}$ into the sets of skew-symmetric
transformations (denoted $\so{}_s$) and symmetric transformations
(denoted $\sym{\E^s}$).

Let $G \subset \Orth{\E^s,\sigma}$ be a compact group of linear
isometries of $(\E^s,\sigma)$. A tangent vector $\xi \in T_1G$ in the
tangent space to the identity of $G$ can be identified with a
matrix. The matrix exponential map restricts naturally to give a map
$\exp : T_1G \to G$. For each $g \in G$, the curve $t \to
\exp(t\xi)\cdot g$ is a curve in $G$ passing through $g$ at $t=0$. Its
derivative $\xi \cdot g$ is therefore a tangent vector in
$T_gG$. Thus, each tangent space is canonically isomorphic to $T_1G$
via right translations.
\footnote{One can equally use left translations.}
\ One typically writes $T_1G = \g$, and calls $\g$ the Lie algebra of
the Lie group $G$. As a set of matrices, $\g$ is equipped with the Lie
bracket denoted by $[\xi,\eta] = \xi \cdot \eta - \eta \cdot \xi$. One
can easily verify that $\xi,\eta \in \g$ implies that $[\xi,\eta] \in
\g$. In addition, for each $g \in G$, $\xi \in \g$, the element
$g\cdot\xi\cdot g^{-1} \in \g$. It is conventional to write $\Ad{g}
\xi = g\cdot\xi\cdot g^{-1}$ and observe that $\Ad{} : G \to \GL{\g}$
is a representation, called the adjoint representation. One knows that
$\left.\frac{\d\ }{\d t}\right|_{t=0}\, \Ad{\exp(t\xi)}\eta =
[\xi,\eta]$, so the derivative $\d_1 \Ad{} =: \ad{} : \g \to \gl{\g}$
is a linear representation of $\g$.

The trace form $(\xi,\eta) \mapsto \trace{\xi'\eta}$ is positive
definite on $\g$. Moreover, the trace form is invariant under the
adjoint representation of $G$, \ie $G$ acts as a group of isometries
of this euclidean structure on $\g$. For a subspace $V \subset \g$,
let $V^{\perp}$ denote its orthogonal complement with respect to the
trace form.

For each $\vtheta\in\E$, let the set $\orb{G}{\vtheta} = \left\{ \phi
\in \E\ :\ \exists g \in G \text{ and } \phi=g\cdot \vtheta \right\}$
be the $G$-orbit of $\vtheta$ and let $\stab{G}{\vtheta} = \left\{ g\in
G\ :\ g\cdot \vtheta=\vtheta \right\}$ be the $G$-stabilizer of
$\vtheta$. It is a well-known theorem that $\orb{G}{\vtheta}$ is a
smooth submanifold of $\E$. The tangent space to $\orb{G}{\vtheta}$ at
$\phi$ can be identified with $\stab{\g}{\phi}^{\perp}$, where
$\stab{\g}{\phi} \subset \g$ is the Lie algebra of
$\stab{G}{\phi}$. Indeed, since $G$ acts transitively, the map $\g
\to T_{\phi}(\orb{G}{\vtheta}) : \xi \mapsto \xi \cdot \phi$ is onto
and its kernel is $\stab{\g}{\phi}$. If $\phi = g \cdot \vtheta$, then
one sees that $\stab{G}{\phi} = g\cdot\stab{G}{\vtheta}\cdot g^{-1}$
and similarly for the Lie algebras.

The normal bundle $\nb[]{\orb{G}{\vtheta}}$ of $\orb{G}{\vtheta}$ is
isomorphic to the vector bundle
\begin{equation} \label{eq:nb}
\nb[]{\orb{G}{\vtheta}} = G \times_{\stab{G}{\vtheta}}
\left(T_{\vtheta}\orb{G}{\vtheta}\right)^{\perp} = G \times_{\stab{G}{\vtheta}}\nb[\vtheta]{\orb{G}{\vtheta}}.
\end{equation}
Here, $G \times \nb[\vtheta]{\orb{G}{\vtheta}}$ is the cartesian product
of the group $G$ with the orthogonal complement
$\nb[\vtheta]{\orb{G}{\vtheta}}$ to the tangent space to $G$'s orbit
through $\vtheta$. The stabiliser $\stab{G}{\vtheta}$ acts linearly on
$\nb[\vtheta]{\orb{G}{\vtheta}}$ and by right translation on $G$. The
set $G \times_{\stab{G}{\vtheta}}\nb[\vtheta]{\orb{G}{\vtheta}}$ is the
quotient space whose points are the sets ($\stab{G}{\vtheta}$-orbits)
$[g,v] = \left\{ (g\cdot h, h \cdot v)\ :\ h \in \stab{G}{\vtheta}
\right\}$ for each $(g,v) \in G \times \nb[\vtheta]{\orb{G}{\vtheta}}.$

It is also a well-known fact that there is an open neighbourhood of
$\orb{G}{\vtheta}$ which is $G$-equivariantly diffeomorphic to an open
neighbourhood $\tube$ of $\orb{G}{\vtheta}$ in
$\nb[]{\orb{G}{\vtheta}}$. See \cite{MT} for generalities and
\cite{Helgason,HJ} for specifics on linear Lie groups.

\begin{center}
\begin{figure}[htb]
{
\psfrag{x}{$\vtheta$}
\psfrag{y}{$\phi$}
\psfrag{nx}{$\nb[\vtheta]{\vTheta}$}
\psfrag{ny}{$\nb[\phi]{\vTheta}$}
\includegraphics[width=5cm, height=4cm]{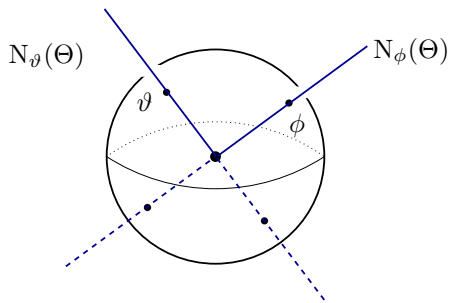}
}
\caption{The tubular neighbourhood $\tube$ and the
  normal bundle $\nb[]{\vTheta}$.} \label{fig:tubular-nhd}
\end{figure}
\end{center}

To simplify notation, $\vTheta$ is used to denote $\orb{G}{\vtheta}$ in
some cases.

\subsection{The projection map onto $\orb{G}{\vtheta}$
} \label{ssec:pj} 
Let us now derive the projection map $\pi : \tube \to
\orb{G}{\vtheta}$. Given $x\in\E$, assume that there is a $g\in G$ such
that 
\begin{align} \label{eq:gx}
g^{-1}x &\in \vtheta + \nb[\vtheta]{\orb{G}{\vtheta}}  &\text{which
  implies } x \in g \vtheta +
\nb[g \vtheta]{\orb{G}{\vtheta}},
\end{align}
by $G$-equivariance. In this case, we can define
\begin{equation} \label{eq:pix}
\pi(x) = g \vtheta.
\end{equation}
\begin{lemma} \label{lem:wd}
There is an open neighbourhood $\tube$ of $\vTheta=\orb{G}{\vtheta}$
such that the map $\pi$ defined in \eqref{eq:pix} is independent of
$\vtheta$ in $\vTheta$. In addition, $\pi$ is a real-analytic
submersion whose fibres are open neighbourhoods of $0 \in
\nb[\phi]{\vTheta}$ for each $\phi \in \vTheta$.
\end{lemma}
\begin{proof}
It suffices to observe that if $x \in \nb[\vtheta]{\orb{G}{\vtheta}}$,
then one can take $g=1 \bmod G_\vtheta$ in \eqref{eq:gx} and
$\pi(x)=\vtheta$, and that (\ref{eq:gx}--\ref{eq:pix}) defines $\pi$ as
a $G$-equivariant map from $\nb[]{\orb{G}{\vtheta}}$ to
$\orb{G}{\vtheta}$. The lemma then follows from the tubular
neighbourhood theorem \cite{MT}. 
\end{proof}

\begin{remark} \label{re:cp}
{\rm 

1/ In general, the affine planes $\vtheta + \nb[\vtheta]{\vTheta}$ and
$\phi + \nb[\phi]{\vTheta}$ intersect each other, as in figure
\ref{fig:tubular-nhd}. At such a point of intersection $\pi$ is not
single-valued; hence these points obstruct the extension of $\pi$ from
a tubular neighbourhood of $\vTheta$ to a globally-defined map on
$\E$. 2/ Lemma \ref{lem:wd} is a consequence of the tubular
neighbourhood theorem for group orbits. Moreover, many
linear-algebraic decompositions are, in fact, an application of this
tubular neighbourhood result.

}
\end{remark}

\subsection{Linear priors on $\orb{G}{\vtheta}$} \label{ssec:priors} 
{
\def\hv{\hat{v}}

As noted in the introduction, the euclidean structure $\sigma$ induces
a linear isomorphism $\E \to \E^* : v \mapsto \hv(\bullet) =
\sigma(v,\bullet)$.
\footnote{One often thinks of this map as $v \mapsto v'$, $v$ maps to
  $v$-transpose. This notation is also used below.}
\ For each $v \in \E$, let $\hv \in \E^*$ be the dual vector induced
by the euclidean structure $\sigma$ and let $f_v = \hv |
\orb{G}{\vtheta}$ be the restriction of $\hv$ to the group orbit. In
terms of the inclusion map $\iota : \orb{G}{\vtheta} \to \E$, one can
write $f_v = \hv \circ \iota$.  Let $\underline{f}_v$ be the minimum
value of $f_v$ and $\bar{f}_v = \int f_v(\phi)\ \d\phi$ be the mean
value of $f_v$ with respect to $\d \phi$, the unique $G$-invariant
probability measure on $\orb{G}{\vtheta}$. (One can define $\bar{\phi}
= \int \iota(\phi)\ \d \phi$ to be the mean element of
$\orb{G}{\vtheta}$, in which case $\bar{f}_v = \left\langle v ,
\bar{\phi} \right\rangle$.  Since $\bar{\phi}$ is a fixed point of
$G$, $\bar{\phi} = 0$ unless $\E$ contains a trivial representation of
$G$.)  Define a bayesian prior density $\lambda=\lambda_v$ by
\begin{equation} \label{eq:lambda_v}
\lambda_v = \alpha f_v + \beta
\end{equation} where the real numbers $\alpha>0$ and
$\beta$ satisfy $\alpha \bar{f}_v + \beta =1$ and $\alpha
\underline{f}_v + \beta = c > 0$.

The chain rule shows that $\d f_v = \d\hv \circ \d \iota$, whence
\begin{align}
\nabla f_v(\phi) &= \d_{\phi}\pi(v) \notag\\
\nabla \log \lambda_v(\phi) &= \frac{\alpha}{\lambda_v} \times
\d_{\phi}\pi(v) && \text{for all }\phi \in
\orb{G}{\vtheta}. \label{al:grad-lambda}
\end{align}
The gradient vanishes at $\phi$ iff $v \in
\nb[\phi]{\orb{G}{\vtheta}}$. The chain rule for second derivatives
shows that $\nabla \d f_v = \nabla \d\hv(\d \iota, \d \iota) + \d\hv
\circ \nabla \d \iota = \d\hv \circ \nabla \d \iota$ since $\nabla \d
\hv = 0$ because $\hv$ is linear. The tensor field $\nabla \d \iota$
is the second fundamental form of $\orb{G}{\vtheta}$ in $\E$ and it is
a measure of the curvature of $\orb{G}{\vtheta}$. Application of the
definition of $\nabla \d \iota$~\cite{EL:1971} shows that $\nabla \d
\iota(\xi_\phi,\eta_\phi) = (1-\d_\phi \pi) \xi \cdot \eta \cdot \phi$
for all $\xi,\eta \in \g$. Thus $\left. \nabla \d f_\vtheta
(\xi,\xi)\right|_\vtheta = -|\xi \cdot \vtheta|^2$ for all $\xi \in
\g_\vtheta^{\perp} \simeq T_\vtheta \orb{G}{\vtheta}$, so the maximum is
a non-degenerate critical point of $f_\vtheta$. A well-known theorem in
Morse theory states that for almost all $v\in\E$, $f_v$ is a Morse
function on $\orb{G}{\vtheta}$ \cite{MT}.

For the purposes of imposing a strong prior, a natural choice is
$v=\vtheta$. The Cauchy-Schwarz inequality plus the fact that $G$ acts
by isometries implies that $f_\vtheta$ (hence $\lambda_\vtheta$) attains
its unique maximum value at $\vtheta$.

Let $\d g$ be the Haar probability measure on $G$. The Haar measure
factors as $\d h \cdot \d \vtheta$, where $\d h$ is the Haar
probability measure on $H$, the stabiliser of $\vtheta$, and $\d
\vtheta$ is the unique $G$-invariant probability measure on
$\vTheta = \orb{G}{\vtheta}$. Define
\begin{equation} \label{eq:vtilde}
\tilde{v}_{\vtheta} = \int_{g \in G} \d g\, \log \lambda_{g\cdot v}(\vtheta) \, g\cdot v
\end{equation}
to be the mean of $v$ over $\orb{G}{v}$ taken with respect to a
peculiar measure.

\begin{thm} \label{thm:linear-priors}
Let $\vTheta = \orb{G}{\vtheta}$ and $\gamma$ be the identity map of
$\vTheta$. Let the bayesian prior density $\lambda=\lambda_v$ be
defined by \eqref{eq:lambda_v}. Let $x \in \tube$, $\hat
\vtheta=\pi(x)$ and $\xi \in \stab{\g}{\hat \vtheta}^{\perp}$ be the
unique vector such that $\xi \cdot \hat \vtheta = \d_{\hat
  \vtheta}\pi(v)$. The bayesian estimator $\gb_{\epsilon}$ and its risk
equal
\begin{align}
\gb_{\epsilon}(x) &= \exp(s \xi + O(\epsilon^4)) \cdot \hat\vtheta &
s=\frac{\alpha \epsilon^2}{\lambda_{v}(\hat\vtheta)} \label{al:linear-priors-g}\\ 
R_{\epsilon}(\gb_{\epsilon};\lambda_v)
&= \epsilon^2 \dim \vTheta + \epsilon^4 \left(
\frac{1}{3}\,\scal_{\vTheta,\lambda} + \left\langle \tilde{v}_{\hat\vtheta}
, \tau(\iota)_{\hat\vtheta} \right\rangle\right) + O(\epsilon^6)
\label{al:linear-priors-r}
\end{align}
where $\exp$ is the exponential map of the Lie group $G$,
$\scal_{\vTheta,\lambda}$ is the average of the scalar curvature of
$\vTheta$ with respect to $\d \vtheta \lambda$ and $\tau(\iota)$ is the
normal vector field of $\vTheta$.
\end{thm}

The proof of \eqref{al:linear-priors-g} applies Theorem \ref{thm:best}
and the fact that $G$ acts as a transitive isometry group of
$\vTheta$. It should be noted that, although $\tilde{v}_\vtheta$ is not
independent of $\vtheta$, the inner product $\left\langle
\tilde{v}_\vtheta , \tau(\iota)_{\vtheta} \right\rangle$ is
independent. In addition, the integration-by-parts formula is needed
to demonstrate \eqref{al:linear-priors-r}.

In the particular case of $v=0$, $\lambda$ is a flat prior density,
the $G$-invariant measure $\d \theta$ is the flat distribution and the
estimator is $\gb_{\epsilon}(x) = \hat \vtheta + O(\epsilon^4)$ with
risk $R_{\epsilon}(\gb_{\epsilon};1) = \epsilon^2\dim \vTheta +
\epsilon^4 \scal_{\vTheta}/3 + O(\epsilon^6)$ where $\scal_{\vTheta}$
is the mean scalar curvature of $\vTheta$. The flat prior produces the
minimax estimator in this case.

\subsubsection{A sample application: $S^2$} \label{sssec:linear-priors-s2-example}
Let us apply theorem \ref{thm:linear-priors} to the case where
$\E=\E^3$, $\vTheta$ is the $2$-dimensional unit sphere $S^2$ in $\E^3$
and $G=\SO{3}$ is the group of linear, orientation-preserving
isometries of $\E^3$. In this case, the projection map
\begin{wrapfigure}[10]{l}[0pt]{0pt}
{
\psfrag{ny}{{\tiny $x$}}
\psfrag{y}{{\tiny $\hat \vtheta$}}
\psfrag{v}{{\tiny $t \cdot \d_{\hat \vtheta}\pi\cdot v$}}
\psfrag{w}{{\tiny $\gb_{\epsilon}(x)$}}
\includegraphics[width=4cm, height=4cm]{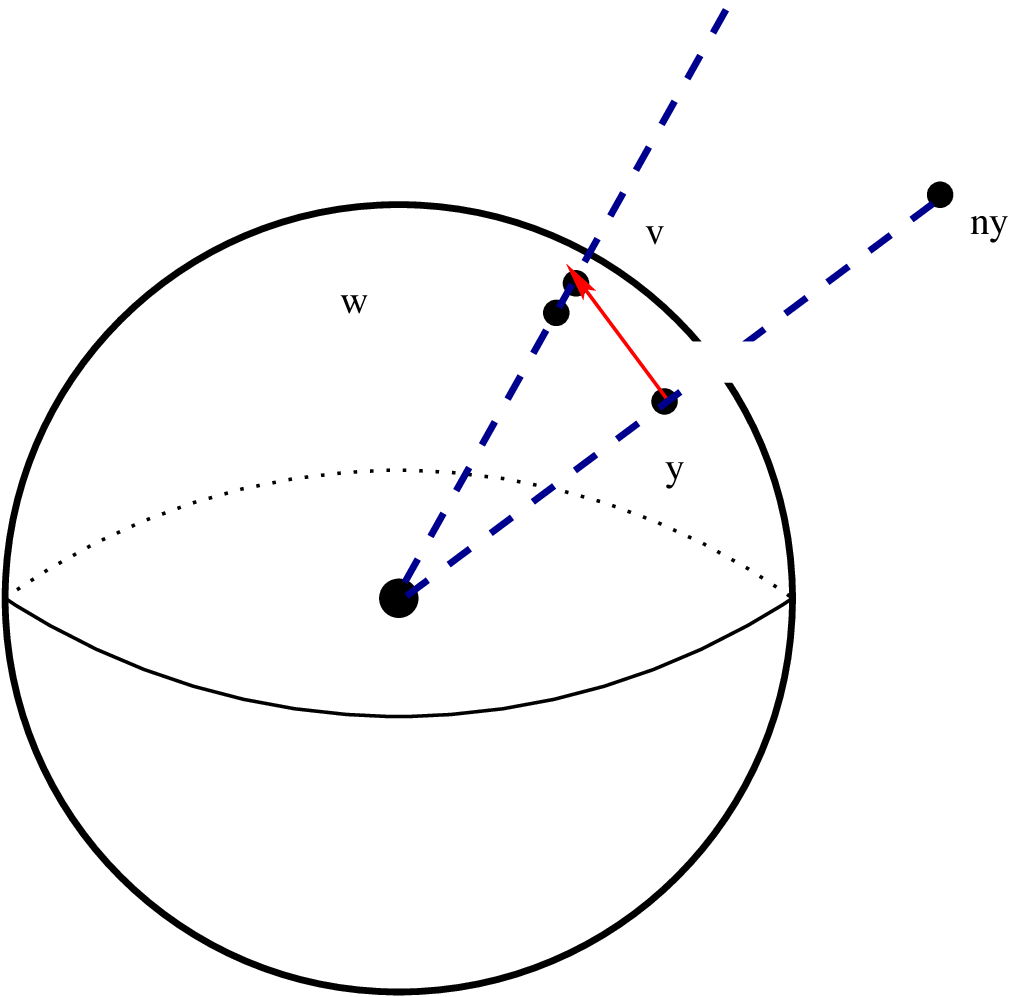}
}
\end{wrapfigure}
is $\pi(x) = x/|x|$ and the projection of $v$ onto $T_{\hat \vtheta}
\vTheta$ is the orthogonal projection $\bar{v} = \d_{\hat \vtheta}\pi(v)
= v - \langle v , \hat \vtheta \rangle\, \hat \vtheta$. The bayesian
estimator in this case is
\begin{align} \label{al:linear-priors-s2-example}
\hspace{40mm}&& \gb_{\epsilon}(x) = \exp(t \hat{v}) \cdot \hat \vtheta + O(\epsilon^4), &&
t=\frac{\epsilon^2 |\bar{v}|}{\langle v , \hat \vtheta \rangle + \beta}
\end{align}
where $\hat{v} \in \so{3}$ is the rotation by $\pi/2$ radians
counterclockwise in the plane orthogonal to $\bar{v}$.

If one supposes that $v \in S^2$, then $\beta=1$ and $0 \leq \alpha
<1$. Since $\tau(\iota)_{\vtheta} = \vtheta$, one computes that
\begin{align} \label{eq:linear-priors-s2-tildev-dot-tau}
\langle  \tilde{v}_v,\tau(\iota)_v \rangle  = (\alpha^{-2}-1) \log
\sqrt{\frac{1-\alpha}{1+\alpha}} + \frac{1}{\alpha}.
\end{align}
\begin{maximacode}
assume(alpha>0);
elambda : alpha*cos(b)+beta;
'integrate(log(elambda)*sin(b)*cos(b),b,0,
changevar(
ev(
ratsimp(
subst(beta=1,
tex(
(alpha^(-2)-1)*log(sqrt((1-alpha)/(1+alpha)))+1/alpha;
tex(
\end{maximacode}
On the other hand, the scalar curvature of the $2$-dimensional unit
sphere is twice the Gaussian curvature, hence is $2$, and the mean of
$\lambda$ is $1$, so $\scal_{\vTheta,\lambda}=2$. The bayesian risk is
therefore
\begin{align} \label{al:linear-priors-s2-br}
R(\gb_{\epsilon};\lambda_v) = 2 \epsilon^2 + \left(\frac{2}{3}+\langle  \tilde{v}_v,\tau(\iota)_v \rangle\right)\, \epsilon^4 + O(\epsilon^6).
\end{align}
Inspection of \eqref{eq:linear-priors-s2-tildev-dot-tau} shows that
the right-hand side is $\alpha$ at $\alpha=0,1$ and it is monotone
increasing on $[0,1]$. This verifies that the flat prior ($\alpha=0$)
yields the second-order minimax estimator.  

}

\subsection{Derivation and application of the projection map} \label{ssec:apps}
This section applies Lemma \ref{lem:wd} to a wide range of orbit
spaces. Lemma \ref{lem:wd} says that to construct the bayesian
estimator of theorem \ref{thm:best} it is necessary to give a concrete
description of a tubular neighbourhood $\tube$ and the projection map
$\pi : \tube \to \vTheta$ from the tubular neighbourhood to the group
orbit $\vTheta$.

It is also necessary to give a concrete description of the exponential
map of the riemannian manifold $(\vTheta,\metric{g}) \subset
(\E^s,\sigma)$. This problem is solved as in Theorem
\ref{thm:linear-priors}, where one uses the linear isomorphism between
$T_{\vtheta} \Theta$ and $\stab{\g}{\vtheta}^{\perp}$, which pulls
back the riemannian exponential map to the Lie group's exponential
map.

\cite[Chapters 1--3]{Chikuse} provide a nice background, aimed at
statisticians, for many applications of several of the orbit spaces
considered below.

\subsubsection{The sphere $S^{n-1}$} \label{sssec:a-1} 
Let $G=\Orth{\E^n}$, and $\vtheta\in\E^n$ be non-zero. The group orbit
$\orb{G}{\vtheta}$ is the sphere of radius $r=|\vtheta|$. Without loss
of generality, one can suppose that $r=1$ and $\E^n$ has a basis $e_i$
where $e_1 = \vtheta$. In this case, $\nb[\vtheta]{\orb{G}{\vtheta}} = \R
\vtheta$ and $g^{-1} x \in \vtheta + \nb[\vtheta]{\orb{G}{\vtheta}}$ iff
$g^{-1} x = \lambda e_1$ iff $x=\lambda\, ge_1$ and $\lambda = \pm
|x|$. Because $\pi$ must be the identity on $\orb{G}{\vtheta}$, one
sees that $\lambda>0$ and therefore $\forall x\neq 0$
\begin{equation} \label{eq:sn}
\pi(x) = g \vtheta = x/|x|.
\end{equation}
In this case $\tube = \E-\left\{ 0  \right\}$.

\subsubsection{The Stiefel manifold} \label{sssec:a-2} 
There is a natural generalisation of the unit sphere introduced by
Steifel~\cite{Chikuse,Husemoller}. Let $v=[v_1 \cdots v_k]$ be a
$k$-tuple of unit vectors $v_i \in \E^n$ for $k\leq n$ which are
mutually orthogonal. The set of all such orthonormal $k$-frames $v$ is
called a Stiefel manifold and denoted by $\stf{k}{\E^n}$. One can
naturally identify $\stf{k}{\E^n}$ as a subset of $\E=\Hom{\E^k,\E^n}$
(the $n \times k$ real matrices). The euclidean structure $\sigma =
\left\langle \cdot , \cdot \right\rangle$ on $\E$ is defined by
\begin{equation} \label{eq:eu}
\left\langle x , y \right\rangle = \trace{x'y}
\end{equation}
for all $x,y \in \E$ where $x'$ is the transpose of $x$. The group
$G=\Orth{\E^n}$ acts on $\E$ by left multiplication and with the frame
$\vtheta = [e_1 \cdots e_k]$
\begin{equation} \label{eq:vkn}
\stf{k}{\E^n} = \orb{G}{\vtheta}.
\end{equation}
Given $x \in \E$, the map
\begin{align} \label{al:k}
\kappa(x) &= x'x & \kappa : \Hom{\E^k,\E^n} \to \sym{\E^k},
\end{align}
from the $k\times k$ matrices to the symmetric $k\times k$ matrices,
defines a submersion when $x$ is of maximal rank $k$ and $G$ acts
transitively on the fibre of $\kappa$. Thus, if $x \in \E$ is of
maximal rank, then the normal space $x+\nb[x]{\orb{G}{x}}$ can be
identified with $x'x + \sym{\E^k}$ via the linearized map $\d_x \kappa$.

To compute the projection map $\pi : \tube \to \stf{k}{\E^n}$: let
$\tube$ be the connected component containing $\vtheta$ of the set of
$x\in \E$ of maximal rank. For each $x \in \tube$, $x'x$ is a
symmetric, positive-definite matrix and therefore $x'x$ has a unique
symmetric positive-definite square root $\tau =:
(x'x)^{\frac{1}{2}}$. Let us define
\begin{align} \label{al:sqrt}
\pi(x) &= x(x'x)^{-\frac{1}{2}} & \pi : \tube \to \stf{k}{\E^n}.
\end{align}
It is clear that $\pi$ is a $G$-equivariant map, $\pi(\vtheta)=\vtheta$
and since $\kappa \circ \pi$ maps $\tube$ to $1 \in \sym{\E^k}$, the
image is $\stf{k}{\E^n}$ and $\pi|_{\stf{k}{\E^n}}=id$. These facts
suffice to show that the map $\pi$ is indeed the projection map of the
normal bundle. (If one had taken another square root of $x'x$ to define
$\pi$, then $\pi(\vtheta) \neq \vtheta$, so that map could not be the
projection map of a tubular neighbourhood).

In the general case, let $\vtheta \in\E$ be of maximal rank and let
$\tau$ be the unique positive-definite symmetric square root of
$\theta' \vtheta$. The projection map $\pi : \tube \to \orb{G}{\vtheta}$
is then
\begin{align} \label{al:sqrt-2}
\pi(x) &= x(x'x)^{-\frac{1}{2}} \tau & \pi : \tube \to \orb{G}{\vtheta},
\end{align}
where $\tube$ is the set of maximal rank elements in $\E$.

One can specialize the above construction to obtain:
\begin{enumerate}
\item[$k=1:$] In this case, $\stf{1}{\E^n}= S^{n-1}$ and
  $(x'x)^{\frac{1}{2}}=|x|$, so \eqref{al:sqrt} specializes to yield
  the projection map onto $S^{n-1}$;
\item[$k=2:$] In this case, $\stf{2}{\E^n}$ is the unit sphere bundle
  of $S^{n-1}$, so \eqref{al:sqrt} specializes to yield the projection
  map from the set of non-collinear vectors in $\E^n \times \E^n$ to the
  unit sphere bundle $S(S^{n-1})$;
\item[$k=n:$] In this case, $\stf{n}{\E^n}= \Orth{\E^n}$ and
  $p=(x'x)^{\frac{1}{2}}$ is the polar factor in the polar
  decomposition $x=gp$ where $g \in \Orth{\E^n}$ and $p \in
  \sym{\E^n}$. Thus \eqref{al:sqrt} specializes to yield the
  projection map of an arbitary invertible matrix onto its orthogonal
  part. See example \ref{sssec:a-9} below for a general construction.
\end{enumerate}

\subsubsection{The real Grassmannian manifold} \label{sssec:a-3} 
Another group orbit space that is closely related to the Steifel
manifold is the manifold of unoriented $k$-dimensional planes in
$\E^n$, called the Grassmannian
manifold~\cite{Chikuse,Husemoller}. Let $\gr{k}{\E^n}$ denote the
Grassmannian manifold of unoriented $k$-planes in $\E^n$. A $k$-plane
$\pl$ in $\E^n$ is uniquely characterized by an orthogonal projection
$\p_\pl \in \Hom{\E^n,\E^n}$ which is symmetric, has an image equal to
$\pl$ and kernel equal to $\pl^{\perp}$. Since each plane $\pl$ and
its orthogonal complement admit an orthonormal basis, we have the
following natural description of the Grassmannian manifold as an orbit
space ($k+l=n$)
\begin{align} \label{al:gr}
\gr{k}{\E^n} &= \orb{G}{\vtheta} & G=\Orth{\E^n},\ \ \vtheta=\begin{bmatrix*}[l]
1_k & 0\\0 & 0_l
\end{bmatrix*}
\intertext{where the action of $G$ on the symmetric matrices
  $\sym{\E^n} \subset \Hom{\E^n,\E^n}$ is by conjugation/congruence}
g\cdot x &= gxg' & \forall g\in G, x\in \sym{\E^n}. \notag
\intertext{The Grassmannian manifold is equivariantly diffeomorphic
  to} \gr{k}{\E^n} & =
\Orth{\E^n}/\Orth{\E^k}\times\Orth{\E^l}. \notag
\end{align}
To identify the normal space $\nb[\vtheta]{\gr{k}{\E^n}}$, a
computation shows that 
\begin{align} \label{al:nb-gr}
T_\vtheta \gr{k}{\E^n} &= \left\{ \begin{bmatrix*}[l]
0_k & \alpha\\ \alpha' & 0_l
\end{bmatrix*}\ :\ \alpha\in\Hom{\E^l,\E^k}  \right\} 
\intertext{whence the normal space is}
\nb[\vtheta]{\gr{k}{\E^n}} &= \left\{ \begin{bmatrix*}[r]
\beta & 0\\ 0 & \gamma
\end{bmatrix*}\ :\ \beta\in\sym{\E^k},\gamma\in\sym{\E^l}   \right\}.
\end{align}
Recall that every $x\in\sym{\E^n}$ is congruent via some $g \in
\Orth{\E^n}$ to a diagonal matrix
$\lambda=\diag{\lambda_1,\ldots,\lambda_n}$ where the eigenvalues
satisfy $\lambda_1 \geq \lambda_2 \geq \cdots \geq \lambda_n$. One may
then define, for all those $x$ with $\lambda_k > \lambda_{k+1}$, the
projection of $x$ onto the Grassmannian manifold via
\begin{align} \label{al:pi-gr}
x = g \lambda g' &&\implies&& \pi(x) &= g \vtheta g'.
\end{align}
Equivalently, $x$ is congruent via a $g_1 \in \Orth{\E^n}$ to a matrix
$y$ in $\nb[\vtheta]{\gr{k}{\E^n}}$ where the eigenvalues of $\beta$
dominate those of $\gamma$. In this case, one can define $\pi(x)=g_1
\vtheta g_1'$. The two definitions of $\pi$ coincide since $g=g_1 \bmod
$ the stabilizer of $\vtheta$. In this case, the tubular neighbourhood
$\tube$ is the connected component containing $\vtheta$ of the set of
$x\in\sym{\E^n}$ which have eigenvalues such that
$\lambda_k>\lambda_{k+1}$.

\begin{remark} \label{re:gr}
{\rm

In the above construction, one may replace
$(\E^n,\sym{\E^n},\Orth{\E^n})$ and the real transpose by
$(\C^n,\sym{\C^n},\U{}_n \subset \Orth{\C^n})$ and the conjugate
transpose (resp.  $(\H^n,\sym{\H^n},\Sp{\H^n} \subset \Orth{\H^n})$
and the quaternionic conjugate transpose) to obtain the grassmannian
of complex $k$-planes in $\C^n$ (resp. the grassmannian of
quaternionic $k$-planes in quaternionic $n$-space $\H^n$). In these
cases, one views $\C^n$ (resp. $\H^n$) as a real euclidean vector
space, where the euclidean structure is provided by the real part of
the hermitian (resp. quaternionic) structure, and the isometries
preserve both the euclidean structure and the complex
(resp. quaternionic) structure. The construction of the projection map
of the tubular neighbourhood is essentially the same. Since the
conclusions of Theorem \ref{thm:best} rely only on the real euclidean
structure, the conclusions remain valid.

}
\end{remark}

\subsubsection{The singular-value decomposition} \label{sssec:a-4}
{
\renewcommand{\p}[0]{{\mathfrak p}}
\def\btheta{\boldsymbol{\vtheta}}

Let $\E=\Hom{\E^k,\E^n}$ (the $n \times k$ real matrices) with the
euclidean structure defined as in \ref{sssec:a-2} and let
$G=\Orth{\E^n} \times \Orth{\E^k}$ act on $\E$ by
\begin{align} \label{al:svd-act}
g \cdot x &= g_1 x g_2^{-1} & \forall g=(g_1,g_2)\in G, x \in \E.
\end{align}
Without loss of generality, one may assume that $k\geq n$. In this
case, the well-known singular-value decomposition says that there is a
$g \in G$ such that $g^{-1} \cdot x = \vtheta$ where $\vtheta$ is in the
``diagonal'' form
\begin{align} \label{eq:svd-diag}
\vtheta &= \left[
\begin{array}{cccc|c}
\vtheta_1 & 0 & \cdots   & 0 & 0\\
0 & \vtheta_2 & \cdots   & 0 & 0\\
\vdots & \vdots & \ddots& \vdots & \vdots\\
0 & 0 & \cdots & \vtheta_n & 0
\end{array}
\right]
&&
\text{and }\vtheta_1 \geq \vtheta_2 \geq \cdots \geq \vtheta_n \geq 0.
\end{align}
 To
 compute the normal space of $\orb{G}{\vtheta}$ is somewhat involved,
 but one can simplify the computation in the following way.

A non-degenerate symmetric bilinear form $\eta$ is of index $(n,k)$ if
it is positive definite (resp. negative definite) on a subspace of
dimension $n$ (resp. $k$). Let $\E^{n,k}$ be a real vector space with
indefinite inner product $\eta$ of index $(n,k)$~\cite{ONeill}. The
orthogonal group, $H=\Orth{\E^{n,k}}$, of this pseudo-euclidean space
is non-compact but its maximal compact subgroup is $G=\Orth{\E^n}
\times \Orth{\E^k}$.
\footnote{In the special case of $k=1$, one has Lorentzian geometry,
  which is important in special relativity \cite{ONeill}.}
\ The Lie algebra, $\h=\orth{}_{n,k}$, of $\Orth{\E^{n,k}}$ contains
the subalgebra $\g=\orth{}_n \oplus \orth{}_k$ and its orthogonal
complement relative to the trace form, the subspace
\begin{align} \label{al:svd-p}
\p &= \left\{ 
\begin{bmatrix*}[l]
0_n & x \\ x' & 0_k
\end{bmatrix*}  \ : \ x \in \Hom{\E^k,\E^n} \right\}.
\end{align}
The action by conjugation of $G=\Orth{\E^n} \times \Orth{\E^k}$ on
$\p$ is naturally identified with the action defined in
\eqref{al:svd-act}.

For $\btheta \in \p$, the orbit of $G$ has
\begin{align} \label{al:svd-tns}
T_{\btheta} (\orb{G}{\btheta}) &= \ad{\g} \btheta &&
\nb[\btheta]{\orb{G}{\btheta}} = \p \cap \left( \ad{\g} \btheta  \right)^\perp =\p_{\btheta }
\end{align}
where $\p_{\btheta}$ is the intersection of the centralizer of $\btheta$
in $\h$ with $\p$. If one supposes that
\begin{align} \label{al:svd-theta}
\btheta &= 
\begin{bmatrix*}[l]
0_n & \vtheta\\ \theta' & 0_k
\end{bmatrix*},
&
 \text{ and $\vtheta$ is in the diagonal form of \eqref{eq:svd-diag}}
\end{align}
then $\p_{\btheta}$ contains all elements of the same form as
$\btheta$. Therefore, for 
\def\x{{\bf x}}
\begin{equation} \label{eq:svd-x}
\x = 
\begin{bmatrix*}[l]
0_n & x\\x' & 0_k
\end{bmatrix*}\in\p
\end{equation}
one knows from the singular-value decomposition of $x$, that $\x$ is
conjugate via a $g\in G$ to an element in $\p_{\btheta}$. Thus, one
has that
\begin{align} 
\pi(\x) &= g \btheta g^{-1} &&  \text{where }
g=(g_1,g_2)\ \  \text{and $g_1^{-1}xg_2$ is diagonal} \label{al:svd-pi-1}\\
\pi &: {\bf T} \to \orb{G}{\btheta} \notag
\intertext{or, equivalently}
\pi(x) &= g_1 \vtheta g_2^{-1} && \pi : \tube \to \orb{G}{\vtheta} \label{al:svd-pi-2}
\end{align}
where $\tube$ is the set of $x$ whose singular values have collisions
nowhere except possibly where those of $\vtheta$ collide and
${\bf T}$ is defined similarly.
\begin{remark} \label{re:svd}
{\rm

From an applied point of view, one may wish to approximate the matrix
$x \in \Hom{\E^k,\E^n}$ by a low rank matrix. To do this, one could
specify $\vtheta$ in \eqref{eq:svd-diag} with $\vtheta_i=1$ for
$i=1,\ldots,l$ and $\vtheta_i=0$ for $i>l$. One can then use
(\ref{al:svd-pi-1}--\ref{al:svd-pi-2}) to project $x$ onto the
low-rank matrix orbit $\orb{G}{\vtheta}$. In applied mathematics, one
approximates $x$ by a low rank matrix in a slightly different
manner. One computes the singular-value decomposition
$\lambda=g_1^{-1}xg_2$ with $\lambda$ in the diagonal form
\eqref{eq:svd-diag}, one truncates $\lambda$ to a diagonal matrix
$\lambda_0$ by zeroing out the singular values $\lambda_{l+1}, \ldots,
\lambda_n$ and then one defines $x_0 = g_1 \lambda_0 g_2^{-1}$ to be
the low-rank approximation (in practice, $l \ll n,k$ so one saves only
the $l$ right and left singular vectors not $g_1$ and $g_2$). In this
case, one knows that the set of rank $l$ matrices is the union over
all rank $l$ orbits, and one uses $x$ to determine the particular
orbit onto which $x$ is projected. By construction, this determines
the rank $l$ orbit that is closest to $x$.

}
\end{remark}
}

\subsubsection{The lagrangian grassmannian} \label{sssec:a-5}
{
\renewcommand{\p}[0]{{\mathfrak p}}
\def\btheta{\boldsymbol{\vtheta}}
\def\x{{\bf x}}

A totally real subspace $\Pi$ of $\C^n$ is a real subspace which has
the distinguished property that $\Pi \cap i \Pi = 0$; a totally real
$n$-plane in $\C^n$ is also called a lagrangian plane. Let $\Lambda_n
= \U{}_n/\Orth{}_n$ be the manifold of lagrangian planes in
$\C^n$. This manifold arose in Maslov's work on quantisation
\cite{Maslov}. It is well-known in hamiltonian mechanics that the
stable and unstable subspaces of a hyperbolic linear hamiltonian
system are both lagrangian planes \cite{AbrahamMarsden}. 

One can embed $\Lambda_n$ into $\E = \sym{\C^n}$, the subspace of
complex $n \times n$ matrices which are symmetric under the transpose
(not conjugate transpose). The euclidean structure
$\sigma=\left\langle \cdot , \cdot  \right\rangle$ on $\E$ is defined
by
\begin{align} \label{al:lg-euc}
\left\langle x , y \right\rangle &= \trace{x^*y} &&\forall x,y \in \sym{\C^n},
\end{align} 
where $x^*$ is the conjugate transpose of $x$. The unitary group
$G=\U{}_n$ acts by isometries of $(\E,\sigma)$ by
\begin{align} \label{al:lg-act}
g\cdot x &= gxg' && \forall g\in\U{}_n, x\in\E.
\end{align}
The stabilizer of $\vtheta=1 \in \E$ under this action is the real orthogonal
subgroup $\Orth{}_n$, so $\Lambda_n = \orb{G}{\vtheta}$.

To compute the projection map, one observes that
$T_\vtheta(\orb{G}{\vtheta}) = i\cdot\sym{\R^n}$ and
$\nb[\vtheta]{\orb{G}{\vtheta}}=\sym{\R^n}$. To define the projection of
$x\in\E$ onto $\orb{G}{\vtheta}$, it is necessary that there exists a
unitary $g\in\U{}_n$ such that $g^{-1}\cdot x \in \vtheta +
\nb[\vtheta]{\orb{G}{\vtheta}} = \sym{\R^n}$. Thus,
\begin{align} \label{al:lg-pix}
\exists g\in\U{}_n, p\in\sym{\R^n} \text{ such that } x&=gpg'
&\implies&& \pi(x)=gg'.
\end{align}
To see the connection with the above description of $\Lambda_n$, note
that if $x \in \sym{\C^n}$ admits a factorization as in
\eqref{al:lg-pix}, then the symmetric quadratic from $q_x(u,v) = u'xv$
is totally real on the real subspace $w\cdot\R^n \subset \C^n$ where
$w=(g')^{-1}$; and conversely, if $q_x$ is totally real on
$w\cdot\R^n$, then $x$ admits a factorization as in
\eqref{al:lg-pix}. Provided that $x$ is non-degenerate, $w \in\U{}_n$
and hence $g$, is uniquely defined up to an element in $\Orth{}_n$. 

\begin{remark} \label{re:lg}
{\rm

As in the singular-value decomposition, there is a natural Cartan
decomposition that is associated with this
example~\cite{Helgason}. Let $H=\Sp{\R^{2n}}$ be the group of
symplectic automorphisms of $\R^{2n}=\C^n$ where the symplectic form
is the skew-symmetric bilinear form
\begin{align} \label{al:lg-sf}
\omega(x,y) &= \im{x^*y} &&\forall x,y\in\C^n.
\end{align} 
The Lie algebra $\h$ of $H$ admits a Cartan decomposition $\h=\k+\p$
where
\begin{align} 
\k &= \left\{ \begin{bmatrix*}[r]
\alpha&\beta\\-\beta&\alpha
\end{bmatrix*}\ :\ \alpha=-\alpha', \beta=\beta'  \right\}, &&
\p = \left\{ \begin{bmatrix*}[r]
a&b\\b&-a
\end{bmatrix*}\ :\ a'=a,b'=b  \right\}
\end{align}
where all matrices are real. The maps $(\alpha,\beta) \mapsto \alpha+i
\beta$ and $(a,b) \mapsto a+ib$ shows that $\k \simeq \un{}_n$ and $\p
\simeq \sym{\C^n}$. The action of $K \simeq \U{}_n$ on $\p$ by
conjugation is identified with the action on $\sym{\C^n}$ by
congruences \eqref{al:lg-act}. We note that, by the theory of Cartan
subalgebras, an $\x=(a,b)\in\p$ may be diagonalized over $\R$ on a
basis that is simultaneously symplectic and orthogonal, that is,
$\x\in\p$ is conjugate to a real diagonal matrix via a unitary
transformation~\cite{Helgason}. This implies the validity of the
decomposition \eqref{al:lg-pix} on the set of $\x \in \p$
($x\in\sym{\C^n}$) which are non-singular.

To compute the projection map $\pi : \tube \to \orb{G}{\vtheta}$ when
$\vtheta$ is in general position, it is most easy to apply
\eqref{al:svd-tns}. 

}
\end{remark}

\begin{remark} \label{re:lg-2}
{\rm
In the above construction, one may replace
$(\E^n,\C^n,\E=\sym{\C^n},\U{}_n=\Orth{\C^n})$ and the real transpose
by $(\C^n,\H^n,\E=\sym{\H^n},\Sp{\H^n}=\Orth{\H^n})$ and the complex
conjugate transpose. The orbit of $\vtheta=1$ is the homogeneous space
$\Sp{\H^n}/\U{\C^n}$, which is the grassmannian of totally complex
$n$-planes in $\H^n$.    
}
\end{remark}
}

\subsubsection{The isotropic grassmannians} \label{sssec:a-6}
{
\renewcommand{\p}[0]{{\mathfrak p}}
\def\btheta{\boldsymbol{\vtheta}}
\def\x{{\bf x}}

There are several distingished orbits in $\sym{\C^n}$ in addition to
the lagrangian grassmannian. From the natural embedding $\R^k \subset
\R^n \subset \C^n$, one obtains the grassmannian manifold of isotropic
(or totally real) $k$-planes in $\C^n$
\begin{align} \label{al:iso-mfld}
\Lambda_{k,n} &= \orb{\U{}_n}{\R^k} = \U{}_n/\Orth{}_k \times
\U{}_l && (k+l=n).
\end{align}
If one defines
\begin{align} \label{al:iso-orb}
\vtheta &=
\begin{bmatrix*}[r]
1_k&0\\0&0_l
\end{bmatrix*}
&&
\Lambda_{k,n} = \orb{\U{}_n}{\vtheta} \subset \sym{\C^n}.
\end{align}
The grassmannian of totally real $k$-planes in $\C^n$ arise naturally
in hamiltonian mechanics. For example, the tangent spaces to an orbit
of the Keplerian $2$-body problem trace a closed curve in
$\Lambda_{k,n}$.

To compute the tubular neighbourhood of $\Lambda_{k,n}$ and its
projection map, remark (\ref{re:lg}) implies that each $x \in
\sym{\C^n}$ is congruent via a $g\in\U{}_n$ to a real diagonal
$\lambda = \diag{\lambda_1,\ldots,\lambda_n}$ with $\lambda_1 \geq
\lambda_2 \geq \cdots \geq \lambda_n$. If $\lambda_k > \lambda_{k+1}$,
then the first $k$ largest eigenvalues and consequently the sum of the
eigenspaces is uniquely determined. One can then define $\pi(x) = g
\vtheta g'$, which is well-defined.

Geometrically, this condition amounts to the following. Remark
(\ref{re:lg}) implies that there is some lagrangian plane $\ell$ on
which $q_x$ is totally real. If the condition on the eigenvalues of
$x$ is satisfied, then there is a unique isotropic $k$-plane $\ell_k
\subset \ell$ such that $\ell = \ell_k \oplus \ell_k^{\perp}$. The
lagrangian $\ell$ and $\ell_k^{\perp}$ are not uniquely defined (if
$q_x$ is degenerate), but $\ell_k$ itself is. The map $x \mapsto
\ell_k$ is the projection map of the tubular neighbourhood.
}

\subsubsection{The manifold of orthogonal complex structures on $\R^{2n}$} \label{sssec:a-7}
{
\renewcommand{\p}[0]{{\mathfrak p}}
\def\btheta{\boldsymbol{\vtheta}}
\def\x{{\bf x}}

A complex structure $J$ on $\R^{2n}$ is a linear map such that
$J^2=-1$; it is orthogonal if $J'=J^{-1}$. A complex structure has
eigenvalues $\pm i$ (repeated $n$ times), so it is necessarily
orientation preserving. The conjugate of an orthogonal complex
structure is also an orthogonal complex structure, and conversely, any
complex structure is conjugate to the standard complex structure by an
orthogonal conjugacy. 

Let us construct the manifold of orthogonal complex structures as a
homogeneous space. Embed $\U{\C^n}$ into $G=\SO{\E^{2n}}$ by
\begin{align} \label{al:oc-em}
x&=\alpha+i \beta \mapsto 
\x=\begin{bmatrix*}[r]
\alpha&\beta\\-\beta&\alpha
\end{bmatrix*}
&&
\forall x\in\Hom{\C^n,\C^n}
\end{align}
where $\alpha$ (resp. $\beta$) is the real (resp. imaginary) part of
$x$. Let $\SO{\E^{2n}}$ act on its Lie algebra $\E = \so{}_{2n}$ by
conjugation. The standard complex structure on $\R^{2n}$ is the element
\begin{align} \label{al:oc-orb}
\vtheta&=\begin{bmatrix*}[r]
0&-1\\1&0
\end{bmatrix*}
&&
\text{ which has the orbit }\orb{G}{\vtheta}=\SO{}_{2n}/\U{}_n
\end{align}
since $\U{}_n$ is the stabilizer group of $\vtheta$.

The tangent and normal space to $\orb{G}{\vtheta}$ at $\vtheta$ are equal to 
\begin{align} \label{al:oc-tns}
T_{\vtheta} (\orb{G}{\vtheta}) &= \ad{\p} \vtheta, &
\nb[\vtheta]{\orb{G}{\vtheta}} = \un{}_n, && \p = \left\{ \begin{bmatrix*}[r]
\delta&\gamma\\\gamma&-\delta 
\end{bmatrix*}\ : \ \delta,\gamma\in\so{}_n  \right\}.
\end{align}
where $\p$ is the orthogonal complement of $\un{}_n$ in
$\so{}_{2n}$. It follows from the fact that every $x \in \so{}_{2n}$
is contained in a Cartan subalgebra that 
\begin{align} \label{al:oc-x}
\exists g\in\SO{}_{2n} \text{ such that  }g'xg&=
\left[
\begin{tabular}{c|c}
\hzero{n}  & $\begin{array}{ccc} \phantom{-}\alpha_1 & & \\ &\ddots&
  \\ &&\phantom{-}\alpha_n \end{array}$\\\hline
$\begin{array}{ccc} -\alpha_1 & & \\ &\ddots&
  \\ &&-\alpha_n \end{array}$ & \hzero{n}
\end{tabular} 
\right] =\alpha,
\end{align}
and the stabilizer of $\alpha$ is contained in $\U{}_n$ provided that
$\det x \neq 0$. Thus, the tubular neighbourhood $\tube$ of
$\orb{G}{\vtheta}$ is the connected component containing $\vtheta$ of
the set of $x \in \so{}_{2n}$ such that $\det x \neq 0$. The
projection map $\pi(x) = g \vtheta g'$ is therefore defined on $\tube$.

\begin{remark} \label{re:oc}
{\rm
One can also define the manifold of unitary quaternionic structures
$J$ on $\C^{2n}$. In this case, the homogeneous space is
$\SU{\C^{2n}}/\Sp{\H^n}$ and the construction is essentially the same as
above.
}
\end{remark}
}

\subsubsection{Adjoint orbits} \label{sssec:a-8}
{
\renewcommand{\p}[0]{{\mathfrak p}}
\def\btheta{\boldsymbol{\vtheta}}
\def\x{{\bf x}}

Let $G \subset H$ be compact Lie groups and let $\g \subset \h$ be
their Lie algebras. The negative Cartan-Killing form on $\h$, $(x,y)
\mapsto -\trace{\ad{x} \cdot \ad{y}}$, defines a $G$-invariant
euclidean structure, where $G$ acts on $\h$ by the adjoint action
(conjugation) \cite{Helgason}. Let $\vtheta \in \h$ and let
$\stab{G}{\vtheta}$ be the stabilizer of $\vtheta$ in $G$, and
$\g_\vtheta$ be its Lie algebra. One has
\begin{align} \label{al:ao-tns}
T_{\vtheta} (\orb{G}{\vtheta}) &= \ad{\g} \vtheta, &
\nb[\vtheta]{\orb{G}{\vtheta}} = \ad{\vtheta}^{-1}\left( \g^\perp \right)
\shortintertext{and}
\orb{G}{\vtheta}&=G/\stab{G}{\vtheta}
\end{align}
One knows that there is an equivariant tubular neighbourhood $\tube$
of $\orb{G}{\vtheta}$ such that the projection map $\pi : \tube \to
\orb{G}{\vtheta}$ is defined. The examples above may be formulated in
these terms.

}

\subsubsection{The group $G$ itself} \label{sssec:a-9}
{
\renewcommand{\p}[0]{{\mathfrak p}}
\def\btheta{\boldsymbol{\vtheta}}
\def\x{{\bf x}}

If $G \subset \Orth{\E^n}$, then $G \subset \Hom{\E^n,\E^n}=\E$. With
the euclidean structure on $\E$ defined as in \eqref{eq:eu}, one
obtains a decomposition
\begin{align} \label{al:g-dec}
\E &= \g + \p = T_1 G + \nb[1]{G}.
\end{align}
The projection map of the tubular neighbourhood $\tube$ can be defined
as: if, given $x\in\E$, there is a unique $g\in G$ such that $g^{-1} x
\in \p$, then $\pi(x):=g$. We see that the projection is a
generalization of the polar decomposition encountered above in the
$k=n$ case of example \ref{sssec:a-2}.

}

\section{Regression problems} \label{sec:reg} 
This section deals primarily with the first-order conditions for a
class of non-linear regression problems. Despite the fact that section
\ref{sec:homog-id} showed the construction of second-order minimax
estimators, the geometry that underlay those constructions is very
similar to that required here. We also give a numerical example of the
derived regressor in the specific setting.

Let $\theta_1,\ldots,\theta_k$ be a collection of design points on a
manifold $\Theta$ and let $y_1,\ldots,y_k$ be a random sample of
points on a manifold $\Lambda$ embedded in a euclidean space as in
diagram \eqref{eq:cd-1}. Let the conditional probability density of
$y$ given $\theta \in \Theta$ be $f(y|\gamma(\theta))$, where $\gamma$
is an unknown map.
\begin{equation}\label{eq:cd-1}
\xymatrix@!R@R=6pt{
&& (\E,\sigma) \ar@{<-}[dd]^{\j}\\
\\
(\Theta,\metric{g}) \ar[rruu]^{\j \cdot \gamma}\ar[rr]^{\gamma} && (\Lambda,\metric{h}),
}
\end{equation}
One is interested in estimating the unknown map $\gamma$ by, say,
minimizing the discrepancy function
\begin{equation} \label{eq:reg-sq-err}
\ssqerr{\gamma} = \frac{1}{k} \ds\sum_{l=1}^k \ell(y_l,\gamma(\theta_l)).
\end{equation}
The loss function $\ell : \Lambda \times \Lambda \to \R$ is assumed to
satisfy
\begin{enumerate}
\item $\ell$ is continuous everywhere and smooth a.e.;
\item $\ell(x,y)=\ell(y,x)$ for all $x,y \in \Lambda$;
\item $\ell(x,y) \geq 0$ for all $x,y \in \Lambda$ and equals $0$ iff
  $x=y$;
\item the hessian $\left. \nabla \d\ell \right|_{\nb[]{\Delta
    \Lambda}}$ is non-degenerate, where $\Delta \Lambda=\left\{
  (x,x)\ : x\in \Lambda \right\}$ is the diagonal.
\end{enumerate}
Natural examples of loss functions include: 1) that induced by the
euclidean structure, $\ell(x,y) = |\j x - \j y|^2$ for all $x,y \in
\Lambda$; and 2) that induced by the intrinsic riemannian distance on
$\Lambda$, $\ell(x,y) = \dist(x,,y)^2$ for all $x,y \in \Lambda$.

If one assumes that the space of admissible maps $\gamma$ is
parameterized by a compact finite-dimensional manifold $\Gamma$, a
solution to this estimation problem is
\begin{equation} \label{eq:reg-soln}
\hat{\gamma} = \argmin \left\{ \ssqerr{\gamma} \ :\ \gamma\in \Gamma\right\}.
\end{equation}

To highlight the geometry and minimise the analysis, it is assumed
throughout that the space of admissible maps $\Gamma$ is a compact
finite-dimensional submanifold of $C^{\infty}(\Theta,\Lambda)$.

\subsection{A first-order condition for the least-squares solution} \label{ssec:reg-foc} 
To state the first-order condition for a solution to
\eqref{eq:reg-soln}, one needs some results from differential
topology. The space $C^{\infty}(\Theta,\Lambda)$ may be equipped with
the structure of a Fr\'echet manifold and one may consider $\Gamma$ as
a smooth submanifold.
\footnote{A Fr\'echet space is a Hausdorff, locally convex vector
  space, with a complete translationally invariant metric
  \cite{Rudin}. A Fr\'echet manifold is a Hausdorff topological space
  with an atlas of smooth coordinate charts into a Fr\'echet space.}
\ There are canonical smooth maps $\ev_i : \Gamma \to \Lambda$ defined
by
\begin{align} \label{al:reg-ev}
\xymatrix{
\Theta \times C^{\infty}(\Theta,\Lambda) \ar[r]^(.7){\ev} & \Lambda \\
\Theta \times \Gamma \ar[u]^{1 \times {\rm incl.}} \ar[ur]^{\ev}\\
\Gamma \ar[u]^{\theta_i \times 1} \ar[uur]_{\ev_i}
}
&&
\text{ where\ \ } \ev(\theta,\gamma) = \gamma(\theta).
\end{align}

\begin{proposition} \label{pr:reg-foc}
The first-order condition for $\gamma \in \Gamma$ to be a minimizer of
$\ssqerr{}$ \eqref{eq:reg-sq-err} is that $\d_\gamma  \ssqerr{} \in
\nb[\gamma]{\Gamma} \subset
T^*_{\gamma}C^{\infty}(\Theta,\Lambda)$. That is,
\begin{align} \label{al:reg-foc}
\left. \ds\sum_{i=1}^k \left( \d_{\lambda_i}\j \cdot \d_\gamma \ev_i
\right)^* \cdot \left( \j(y_i) - \j\ev_i(\gamma)\right)
\right|_{T_\gamma \Gamma}  &= 0 & \text{ where }
\lambda_i=\gamma(\theta_i).
\end{align}
\end{proposition}

\begin{remark} \label{re:pr:reg-foc}
{\rm

We have seen above that compact group orbits are important examples of
smooth manifolds; and each of these lie within a sphere of constant
radius. Thus, if $\Lambda$ is contained in a sphere of constant
radius, proposition \ref{pr:reg-foc} yields the first-order condition
\begin{align}
\left.\ds\sum_{i=1}^k \left( \d_{\lambda_i}\j \cdot \d_\gamma \ev_i
\right)^* \cdot \j(y_i) \right|_{T_\gamma \Gamma}  &=0.
\end{align}

}
\end{remark}

\begin{proof}[Proof of Proposition \ref{pr:reg-foc}]
Let us recall the definition of a tangent vector $v \in T_{\gamma}
C^{\infty}(\Theta,\Lambda)$. One may view $v$ as the derivative at
$t=0$ of an equivalence class of smooth curves $\gamma_t$ with
$\gamma_{t=0}=\gamma$. For each $\theta\in \Theta$, $\gamma_t(\theta)$
is a smooth curve on $\Lambda$ through $\gamma(\theta)$. Thus, a
tangent vector $v \in T_{\gamma} C^{\infty}(\Theta,\Lambda)$ is a
smooth map $v : \Theta \to T \Lambda$ such that $v(\theta) \in
T_{\gamma(\theta)} \Lambda$ for all $\theta$ (differential geometers
say that $v$ is a smooth section of $\gamma^*T \Lambda$). It follows
that if $v \in T_\gamma \Gamma$, then $\d_\gamma \ev_i \cdot v$ is a
tangent vector in $T_{\gamma(\theta_i)} \Lambda$.

If $v \in T_\gamma \Gamma$ is a tangent vector, then the chain rule
shows that
\begin{align} \label{al:reg-dsq}
\d_\gamma \ssqerr{}\cdot v &= \frac{2}{k} \ds\sum_{i=1}^k \left\langle \j(y_i)
, \d_{\lambda_i}\j \cdot \d_\gamma \ev_i \cdot v \right\rangle -
\left\langle \j \ev_i(\gamma), \d_{\lambda_i}\j \cdot \d_\gamma \ev_i \cdot v \right\rangle\notag\\
&= \frac{2}{k} \ds\sum_{i=1}^k \left\langle (\d_{\lambda_i}\j \cdot
\d_\gamma \ev_i)^* \cdot \left( \j(y_i) - \j\ev_i(\gamma) \right) ,  v \right\rangle,
\end{align}
where one uses the fact that the euclidean structure on $\E$ allows
one to identify $T_{\bullet}\E$ and $T^*_{\bullet}\E$. This yields
\eqref{al:reg-foc}.
\end{proof}

\subsubsection{Least-squares for linear maps} \label{sssec:lin-lsq}

Assume that $\Theta,\Lambda$ are isometrically embedded in euclidean
spaces $\E_0,\E_1$ with inclusion maps $\iota,\j$ respectively. 

\begin{definition} \label{de:lin-map}
Let $\Gamma \subset C^{\infty}(\Theta,\Lambda)$. One says that
$\Gamma$ is a set of {\em linear maps} if there is a subset $\Delta
\subset \Hom{\E_0,\E_1}$ that map $\Theta$ into $\Lambda$ such that
$\Delta \iota = \j \Gamma$.
\end{definition}

In other words, $\Gamma$ is a set of linear maps if, for each $\alpha
\in \Gamma$, there is an $A \in \Delta$ such that the following
commutes:
\begin{equation} \label{eq:cd-2}
\xymatrix{
\E_0 \ar[r]^{A} & \E_1\\
\Theta \ar[r]^{\alpha} \ar[u]^{\iota} & \Lambda \ar[u]_{\j}.
}
\end{equation}

Since $\j$ is a smooth embedding (resp. $\iota$ is a smooth immersion
when $\Theta$ spans $\E_0$), it is permissible to abuse notation and
identify $\Gamma$ and $\Delta$ as smooth manifolds.

\begin{corollary} \label{cor:reg-foc}
Let $\Theta \subset \E_0$ be a spanning set and let $\Gamma$ be a
submanifold of linear maps. If $\gamma$ is a least-squares solution to
\eqref{eq:reg-soln}, then the matrices
\begin{align} \label{al:reg-lsq-nt}
\nu &=\ds\sum_{i=1}^k y_i \otimes \theta_i', & \tau&=\ds\sum_{i=1}^k
\theta_i \otimes \theta_i'\\
\intertext{ satisfy }
\nu &\equiv \gamma \tau && \bmod \nb[\gamma]{\Gamma},  \label{al:reg-lsq}
\end{align}
where $T_\gamma\Hom{\E_0,\E_1} = T_\gamma \Gamma \oplus \nb[\gamma]{\Gamma}$ as
in \eqref{eq:pix}, and $\nu \in \Hom{\E_0,\E_1}, \tau \in \Hom{\E_0,\E_0}$.
\end{corollary}

The condition \eqref{al:reg-lsq} specialises to the least-squares
regression formula when $\Lambda=\E^1$, $\Theta=\E^s$ and $\Gamma$ is
the space of linear functions $\Hom{\E^s,\E^1}$ so that the normal
space is trivial. In the usual least-squares regression formula, the
coefficient vector is viewed as a column vector, whereas here one
views the coefficient vector as a linear function and hence a row
vector. The standard formula is recovered by tranposing the normal
equations \eqref{al:reg-lsq}.

An especially useful application of corollary \ref{cor:reg-foc} is
when $\Theta=\Lambda$ and $\Gamma \subset \Hom{\E}$ is a group. In
this case the first-order condition simplifies to
\begin{align}
\nu \gamma' &\equiv \gamma \tau \gamma ' &&\bmod \nb[1]{\Gamma},  \label{al:reg-lsq-gp}
\intertext{and when $\Gamma \subset \Orth{\E}$, since $\tau$ is symmetric,}
\gamma'\nu &\equiv 0 &&\bmod \nb[1]{\Gamma}. \label{al:reg-lsq-on}
\end{align}
In other words, $\gamma$ is the orthogonal projection onto $\Gamma$ of
the matrix $\nu$.

\begin{remark} \label{re:reg-lsq}
{\rm

Kim \cite{Kim:1991} looks at the spherical regression problem where
one has $n$ known design points $x_i$ on $S^2$ and there are $n$
observations $y_i$ on $S^2$ which are distributed about $\alpha x_i$
where $\alpha \in \SO{\E^3}$ is unknown. For a uniform bayesian prior
on $S^2$ and the discrepancy function $\ssqerr{a} = \frac{1}{n}
\ds\sum_{i=1}^{n} |y_i - a x_i|^2$, Kim shows that the bayesian
estimator is the ``least-squares'' estimator obtained as follows. Let
\begin{align} 
x &= \frac{1}{n} YX' && \text{ where }Y=[y_1 \cdots y_n], X=[x_1 \cdots x_n]
\notag\\
&= u \sigma v' && u,v \in \SO{\E^3},
\sigma=\diag{\sigma_1,\sigma_2,\sigma_3}\label{al:stf-svd}\\
&&&\text{ with } \sigma_1 \geq \sigma_2 \geq \sigma_3.\notag
\shortintertext{The least squares estimator is then}
\hat{\alpha} &= uv'. \label{al:stf-a}
\end{align} 
If one observes that the singular-value-like decomposition of $x$ in
\eqref{al:stf-svd} can be rewritten to obtain a polar-like
decomposition of $x$,
\begin{align} \label{al:stf-b}
x &= uv'(v \sigma v') = gp && g=uv' \in\SO{\E^3}, p= v \sigma v' \in
\sym{\E^3} \notag
\shortintertext{ whence }
\hat{\alpha} &= \pi(x) && \text{ provided } \sigma_3 > 0.
\end{align}
One can see that \eqref{al:reg-lsq-on} generalises
Kim's~\cite{Kim:1991} formula for the spherical regression problem.

}
\end{remark}

\begin{proof}[Proof of Corollary \ref{cor:reg-foc}]
Since $\Gamma$ acts linearly on $\E_0$, one sees that the normal
equations \eqref{al:reg-foc} are linear in $\gamma$ and simplify to: for all
$v \in T_{\gamma}\Gamma$
\begin{align}
0 &= \sum_{i=1}^k \left\langle y_i - \gamma \cdot \theta_i , v\cdot
\theta_i \right\rangle
= \sum_{i=1}^k \left\langle y_i \otimes \theta_i' - \gamma \cdot
\theta_i \otimes \theta_i' , v \right\rangle \label{al:reg-dev}
\end{align}
where the second inner product is the trace inner product as in
\eqref{eq:eu} and the inclusion map $\j$ is dropped to simplify
notation. Rearranging \eqref{al:reg-dev} yields \eqref{al:reg-lsq}.

To arrive at \eqref{al:reg-lsq-on}, one notes that when $\Gamma
\subset \Hom{\E,\E}$ is a group, then each tangent vector $v \in
T_\gamma \Gamma$ is of the form $v=\xi \cdot \gamma$ where $\xi \in
T_1 \Gamma$. The normal equations \eqref{al:reg-dev} are then
transformed to
\begin{align}
0 &= \sum_{i=1}^k \left\langle y_i \otimes \theta_i' \cdot \gamma' - \gamma \cdot
\theta_i \otimes \theta_i' \cdot \gamma' , \xi \right\rangle \label{al:reg-dev-gp}
\end{align}
for all $\xi \in T_1 \Gamma$.  Rearranging \eqref{al:reg-dev-gp}
yields \eqref{al:reg-lsq-gp}.
\end{proof}

\subsubsection{Regression with the intrinsic
  distance} \label{sssec:reg-id}

Let $\dist : \Lambda \times \Lambda \to \R$ be the riemannian distance
of the riemannian manifold $(\Lambda,\metric{h})$. One can define the
discrepancy functional \eqref{eq:reg-sq-err} using $\dist$ to be
\begin{equation} \label{eq:reg-id-sq-err}
\ssqerr{\gamma} = \frac{1}{k} \sum_{l=1}^k \dist(y_l,\gamma(\theta_l))^2,
\end{equation}
for $\gamma \in C^{\infty}(\Theta,\Lambda)$.

For each $y \in \Lambda$, the function $x \mapsto \dist(y,x)^2$ is
smooth on the open set of $x$ such that there is a unique minimising
geodesic from $y$ to $x$. The set of $x$ on which this function is not
differentiable is the cut locus of $y$ --- a closed, nowhere dense
subset of $\Lambda$. If $x$ is not in the cut locus of $y$, then there
exists a unique shortest tangent vector $w=w_y(x) \in T_x \Lambda$
such that $\exp_x w = y$ and $|w| = \dist(y,x)$. One may write $w =
\log_x y$; one knows that $w=w_y(x)$ is a smooth vector field off the
cut locus of $y$.

\begin{proposition} \label{pr:d-dist}
For $x$ in the complement of the cut locus of $y$,
\begin{equation} \label{eq:d-dist}
\didi{x}{\dist(y,x)^2} = -2\, \log_x y,
\end{equation}
where $T_x^* \Lambda$ and $T_x \Lambda$ are identified via the metric
$\metric{h}$.
\end{proposition}

\begin{proof}
Let $x_t$ be a smooth curve such that $x_{t=0} = x$ and $v = \ds
\left.\ddt{t}\right|_{t=0} x_t$. Let $c_t(s)$ be the unique minimal
geodesic from $x_t$ to $y$. It is clear from figure \ref{fig:jacobi}
that the derivative of $\frac{1}{2}\,\dist(y,x_t)^2$ is $\left\langle
c'(1) , v \right\rangle$ where $c=c_0$. Since there is a unique
shortest geodesic joining $y$ to $x$, reversibility shows that $c'(1)
= -w_y(x) = - \log_x y$.
\end{proof}

\begin{center}
\begin{figure}[htb]
{
\psfrag{a}{$y$}
\psfrag{aa}{$\alpha$}
\psfrag{b}{$x$}
\psfrag{cd}{$c'(1)$}
\psfrag{v}{$v$}
\psfrag{d}{$x_t$}
\psfrag{u}{\hspace{-4mm}\vspace{1mm}geodesic $c(s)$}
\psfrag{w}{\hspace{-23mm}geodesic at time $t$:\ \ \ \ $c_t(s)$}
\includegraphics[width=5cm, height=4cm]{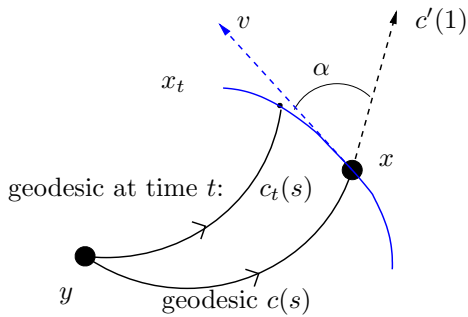}
}
\caption{The derivative of $\dist(y,x)^2$.} \label{fig:jacobi}
\end{figure}
\end{center}

\begin{proposition} \label{pr:reg-id-foc}
If $\gamma \in \Gamma$ is a minimiser of $\ssqerr{}$
\eqref{eq:reg-id-sq-err}, then {\em either} $\ssqerr{}$ is
differentiable at $\gamma$ and 
\begin{equation} \label{eq:reg-id-foc}
\d_\gamma \ssqerr{} = \left. -\frac{2}{k}\, \sum_{l=1}^k \left( \d_\gamma
\ev_l  \right)^*\, \log_{\gamma(\theta_l)} y_l \right|_{T_\gamma
  \Gamma} = 0;
\end{equation}
{\em or} there is an $l$ such that $\gamma(\theta_l)$ lies in the cut
locus of $y_l$.
\end{proposition}

In the first case where $\ssqerr{}$ is defined via the extrinsic
distance \eqref{eq:reg-sq-err}, proposition \ref{pr:reg-foc} results
in a closed form solution for the minimising estimator in many
interesting cases. The intrinsic distance leads to a system of normal
equations which, even in simple cases, appear opaque. However, there
is additional information which one may obtain from these
equations. In the first case, since $\Gamma$ is a finite-dimensional
manifold, let us equip it with some riemannian metric. It is
well-known that the hessian of a smooth function may be defined using
riemannian structures, but that this hessian at a critical point is
independent of those structures. Thus, if one lets $\phi_l(x) =
\dist(y_l,x)^2$, then the calculus of second derivatives gives
\begin{equation} \label{eq:reg-id-hessian}
\left.\nabla \d \ssqerr{}\right|_{\gamma} = \frac{1}{k}\, \sum_{l=1}^k \nabla \d
\phi_l(\d_{\gamma}\ev_l,\d_{\gamma}\ev_l) + \d_{\gamma(\theta_l)} \phi_l \cdot \nabla \d\,\ev_l
\end{equation}
where $\nabla \d \phi_l$ is the hessian of $\phi_l$, etc.. One knows
that $\nabla \d \phi_l(v,w)$ is the second variation of the energy
functional $E[c] = \int_0^1 |c'(s)|^2\, \d s$ along the Jacobi fields
determined by $v,w \in T_{\gamma(\theta_l)} \Lambda$ and the
minimising geodesic $c$ from $y_l$ to $\gamma(\theta_l)$. If it is
assumed that these do not lie in the cut locus of the other, then this
second variation is necessarily positive. Thus, the only way for
$\nabla \d \ssqerr{}$ to not be positive definite is for one or more
of the forms $\d \phi_l \cdot \nabla \d\, \ev_l$ to be negative definite
along some subspace. This cannot happen if $\nabla \d\, \ev_l$ vanishes
for all $l$.

\begin{proposition} \label{pr:reg-id-soc}
If there is a riemannian structure on $\Gamma$ such that $\nabla \d\,
\ev_l = 0$ for all $l$, and $\gamma \in \Gamma$ is a smooth critical
point of $\ssqerr{}$, then $\gamma$ is a local minimum.
\end{proposition}

\begin{example} \label{ex:reg-id}
{\rm

Let $\Phi=\Lambda$ be the unit sphere $S^2$ in $\E^3$ and let
$\Gamma=\SO{3}$ be the group of orientation-preserving isometries of
$S^2$. In this case, the distance function is the angle between
vectors
\begin{align} 
\dist(y,x) &= \arccos(\alpha)
 \label{al:reg-id-s2-dist} \\
\intertext{while the inverse to the exponential function is}
\log_xy &= \frac{\alpha}{\sin \alpha} \, (x \wedge y) \wedge x && \text{where}
&\cos \alpha = \left\langle y , x \right\rangle, 
\label{al:reg-id-s2-log}\\
\end{align}
and $y \neq -x$ in \eqref{al:reg-id-s2-log}. One computes that for
each $\gamma \in \SO{3}$ and $\xi \in \so{3}$
\begin{align}
\left.\ddt{t}  \right|_{t=0}\, \ssqerr{\gamma e^{t\xi}} &= -\frac{2}{k} \sum_{l=1}^k
\trace{ \log_{\theta_l}(\gamma' y_l) \otimes \theta_l' \cdot \xi
} \notag\\
\intertext{so the pullback of $\d_\gamma \ssqerr{}$ to $T_1\SO{3}$ is}
2\tau &= \left.\frac{2}{k} \sum_{l=1}^k \log_{\theta_l}(\gamma' y_l) \otimes
\theta_l' \right|_{\so{3}}. \notag
\end{align}
Thus, $\d_\gamma \ssqerr{}$ vanishes iff $\tau$ is a symmetric
matrix. But modulo $\nb[1]{\SO{3}}=\sym{\E^3}$ one has
\begin{align}
\tau &= \gamma' \nu, 
&&
\text{ where } \nu = \frac{1}{k} \sum_{l=1}^k \frac{\alpha_l}{\sin \alpha_l} y_l
\otimes \theta_l' \label{al:reg-id-s2-foc}
\end{align}
so one concludes that the first-order condition is that $\gamma$ is
the orthogonal projection of $\nu$ onto $\SO{3}$. This is similar to
the least-squares condition \eqref{al:reg-lsq} --- except that in
\eqref{al:reg-id-s2-foc} the matrix $\nu$ is a function of $\gamma$
through the angles $\alpha_l$. However, if one writes $y_l = \gamma'
\theta_l + \epsilon r_l$, and expands the matrix $\nu$ in the small
parameter $\epsilon$, then one has $\nu = \epsilon \nu_1 +
O(\epsilon^2)$ and $\nu_1$ formally is the same as $\nu$ in
\eqref{al:reg-lsq-nt}. In other words, the intrinsic-distance
regressor is a perturbation of the least-squares regressor.

In this example, $\ev_l(\gamma) = \gamma(\theta_l)$ so the map $\ev_l
: \SO{3} \to S^2$ is the canonical projection map. In particular, this
map has vanishing hessian -- $\nabla \d\,\ev_l = 0$ -- so proposition
\ref{pr:reg-id-soc} implies that a smooth solution to the first-order
condition $\tau \equiv 0 \bmod \sym{\E^3}$ is a local minimiser of
$\ssqerr{}$. Moreover, one knows

{
\def\nonsmooth{{\mathfrak N}}
\begin{lemma} \label{le:reg-id-s2-smooth}
Let $\nonsmooth$ be the set of $\gamma$ at which $\ssqerr{}$ is not
differentiable. Then $\nonsmooth$ is a union of translates of
subgroups isomorphic to $\SO{2}$.

If $\gamma$ is a local minimum point of $\ssqerr{}$, then $\ssqerr{}$
is differentiable at $\gamma$. In particular, the regression estimator
$\gamma$ satisfies the property that 
\begin{align} \label{al:reg-id-foc-smooth}
\gamma ' \nu &\equiv 0 && \bmod{\nb[1]{\SO{3}}}
\end{align}
where $\nu$ is defined in \eqref{al:reg-id-s2-foc}.
\end{lemma}

\begin{proof}
Since $\dist(y,x)$ is differentiable in $x$ on the set $S^2 -\{-y\}$,
$\ssqerr{}$ is differentiable at $\gamma$ iff $\gamma(\theta_l) \neq
-y_l$ for all $l$. Thus, if $\ssqerr{}$ is not differentiable at
$\gamma$, then there is a $y_{l_0}$ such that $\gamma(\theta_{l_0}) =
-\theta_{l_0} = y_{l_0}$. If $\gamma_{l}$ is some solution to
$\gamma(\theta_l)=-y_l$, then the set of all solutions to the latter
is $\gamma_l \cdot \stab{}{\theta_l}$, which is a translate of a group
isomorphic to $\SO{2}$. Thus $\nonsmooth = \cup_{l=1}^k \gamma_l \cdot
\stab{}{\theta_l}$.

Let $\gamma$ be a local minimum point of $\ssqerr{}$. Assume that
$\gamma \in \nonsmooth$. Without loss of generality, it can be assumed
that there is an $l_0>0$ such that $\gamma(\theta_l) = -y_l$
(resp. $\gamma(\theta_l) \neq -y_l$) for $l \leq l_0$
(resp. $l>l_0$). 

Let $\ssqerr{}_0$ (resp. $\ssqerr{}_1$) be the part of $\ssqerr{}$
contributed for $l \leq l_0$ (resp. $l>l_0$). Then,
\begin{align}
\ssqerr{}_0(\gamma) &= l_0 \pi^2, && d_{\gamma}\ssqerr{}_1 = 0.
\end{align}
Moreover, if $y,x\in S^2$, then since $y,-y,x$ is a degenerate
triangle, $\dist(y,x) = \dist(y,-y)-\dist(-y,x) = \pi - \dist(-y,x)$.
Therefore, one knows that 
\begin{align}
\ssqerr{}_0(\hat\gamma) &= \sum_{l\leq l_0} (\pi - \dist(-y_l,\hat
\gamma(\theta_l)))^2 \\
&= l_0\pi^2 - 2\pi  \sum_{l\leq l_0} \dist(\gamma(\theta_l),\hat\gamma(\theta_l))
+ O(|\gamma-\hat \gamma|^2),\\
\ssqerr{}_1(\hat \gamma) &= \ssqerr{}_1(\gamma) + O(|\gamma-\hat \gamma|^2).
\end{align}
Since the orbit map $\ev_l : \SO{3} \to S^2$ is a riemannian
submersion, there are $\hat \gamma$ such that, for a fixed $l$,
$\dist(\gamma(\theta_l),\hat\gamma(\theta_l)) = \dist(\gamma,\hat
\gamma)$. This implies that $\ssqerr{}_0$ decreases along $\hat
\gamma$ more than $\ssqerr{}_1$ increases. But $\gamma$ is a local
minimum. Absurd. Therefore, if $\gamma$ is a local minimum, then
$\ssqerr{}$ is differentiable at $\gamma$.
\end{proof}

Let $R_j(s)$ be the counterclockwise rotation of $\E^3$ by $s$ radians
in the plane orthogonal to the $j$-th standard basis vector. Elements
of $\SO{3}$ may be parameterised in terms of `3-1-3' Euler angles:
$\gamma = R_3(a)R_1(b)R_3(c)$ where $a,c \in [0,2\pi]$ and $b \in
[0,\pi]$ \cite{Arnold}. In figure \ref{fig:reg-id-s2-cloud}, one has
an empirical distribution of the regressor $\hat \gamma = \hat
\gamma(\y)$ in Euler angles.  For $k=100$ design points $\theta_l$,
drawn from the uniform distribution on $S^2$, $y_l = u_l/|u_l|$ where
$u_l= \gamma \theta_l + \sigma \cdot \epsilon_l$ and $\epsilon_l$ is
an i.i.d. gaussian in $\E^3$. $N=1000$ draws are made and the
first-order condition \eqref{al:reg-id-s2-foc} is numerically solved
for $\sigma = 0.1$ to $0.9$ in increments of $0.1$. All computations
are performed in \octave\cite{octave}. The starting point for the
numerical solution of \eqref{al:reg-id-s2-foc} is provided by the
orthogonal projection of $\sum_{l=1}^k y_l \otimes \theta_l'$ onto
$\SO{3}$.

{
\def\x{{\bf x}}
Figure \ref{fig:reg-id-s2-distribution} shows the histograms of the
normalised empirical distributions of the Euler angles of the
regressor $\hat \gamma=\hat \gamma(\y)$ and reports the
Kolmogorov-Smirnov p-value for normality. The normalised Euler angles
are of the form $\xi = C^{-1} \x$, where $\x$ is the regressor's Euler
angle, and the sample covariance matrix of $\x$ is $CC'$.
}

\begin{center}
{
\psfrag{mu}{{\tiny mean}}
\psfrag{a}{{\tiny $a$}}
\psfrag{b}{{\tiny $b$}}
\psfrag{c}{{\tiny $c$}}
\psfrag{sigma=0.1}{{\tiny $\sigma=0.1$}}
\psfrag{sigma=0.2}{{\tiny $\sigma=0.2$}}
\psfrag{sigma=0.3}{{\tiny $\sigma=0.3$}}
\psfrag{sigma=0.4}{{\tiny $\sigma=0.4$}}
\psfrag{sigma=0.5}{{\tiny $\sigma=0.5$}}
\psfrag{sigma=0.6}{{\tiny $\sigma=0.6$}}
\psfrag{sigma=0.7}{{\tiny $\sigma=0.7$}}
\psfrag{sigma=0.8}{{\tiny $\sigma=0.8$}}
\psfrag{sigma=0.9}{{\tiny $\sigma=0.9$}}
\psfrag{sigma=1}{{\tiny $\sigma=1$}}
\begin{figure}[htb]
\begin{tabular}{ccc}
\includegraphics[width=4cm, height=4cm]{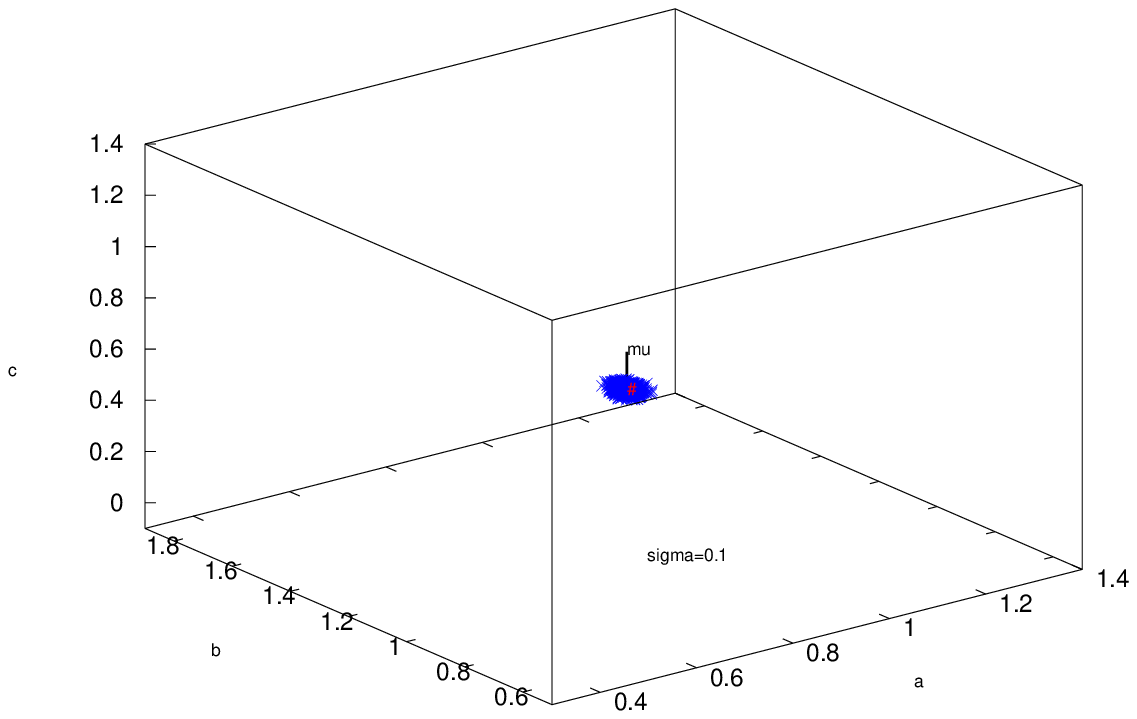} &
\includegraphics[width=4cm, height=4cm]{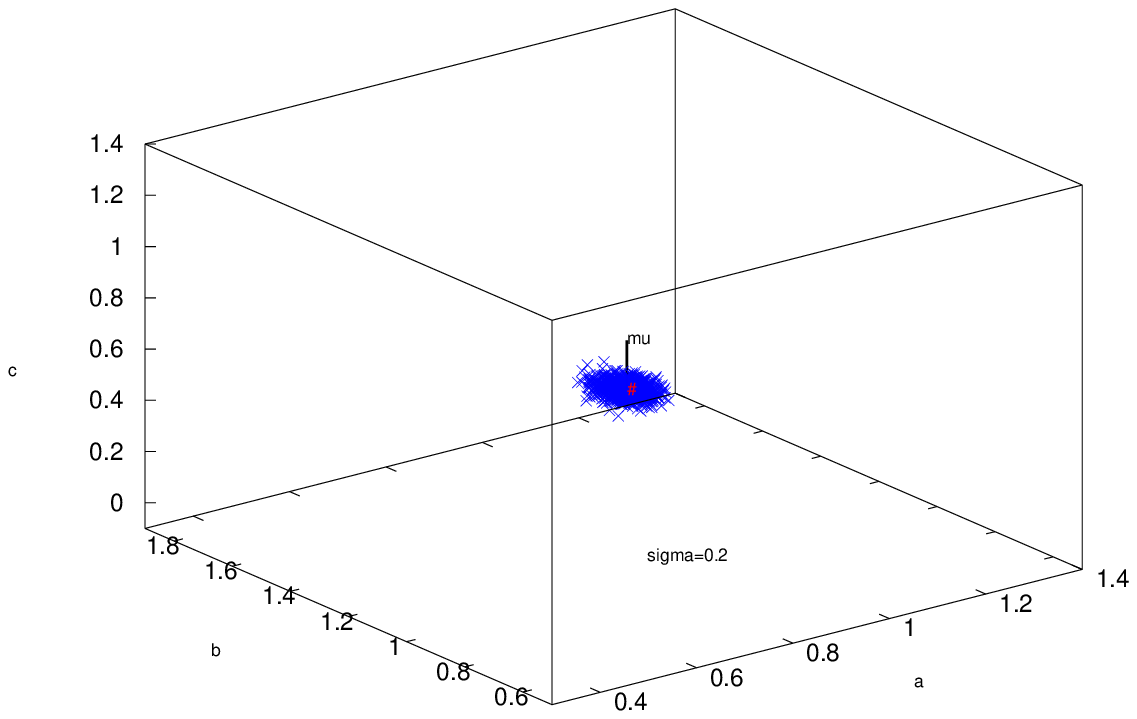} &
\includegraphics[width=4cm, height=4cm]{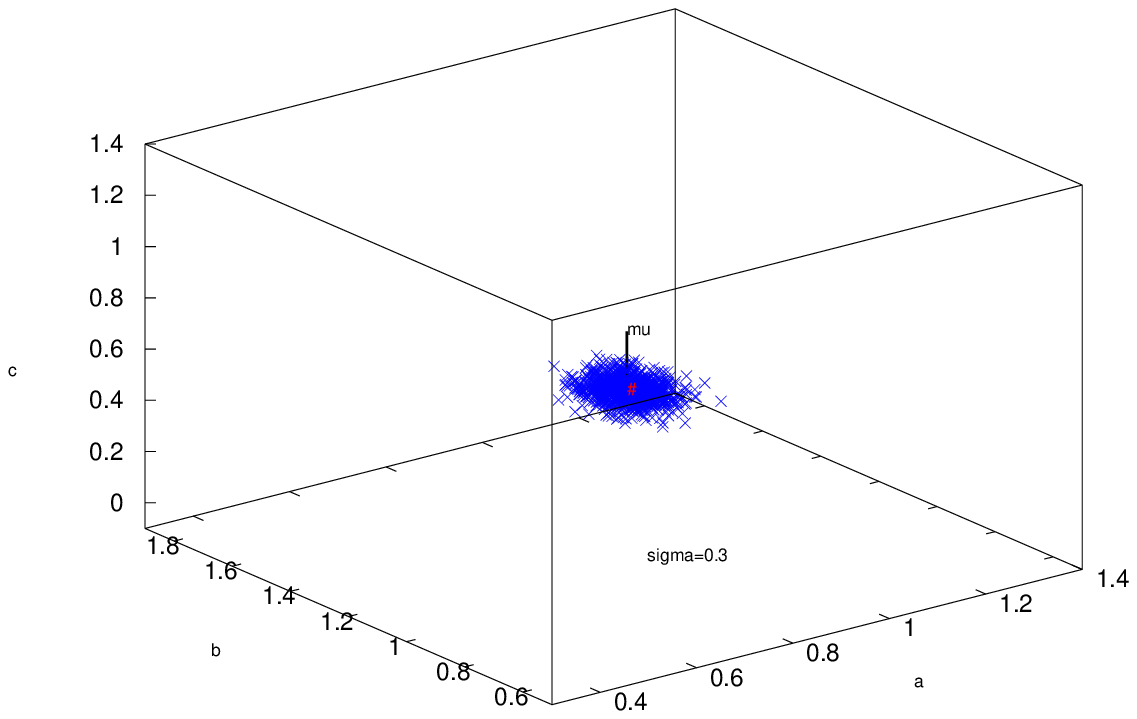} \\
\includegraphics[width=4cm, height=4cm]{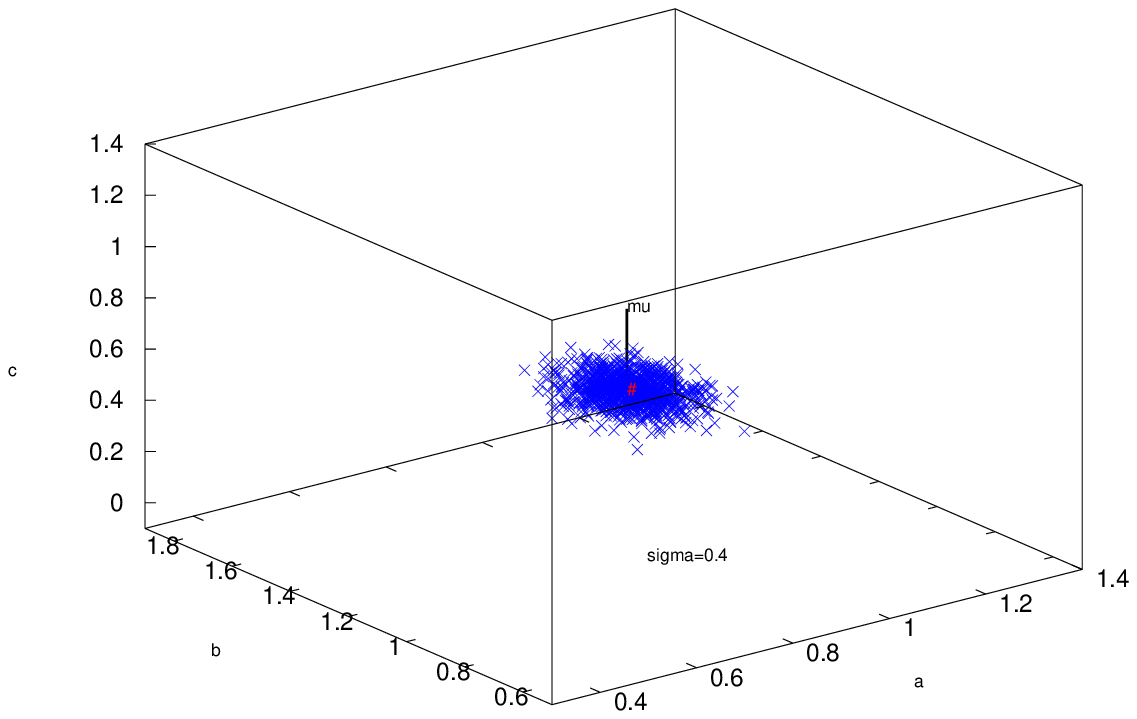} &
\includegraphics[width=4cm, height=4cm]{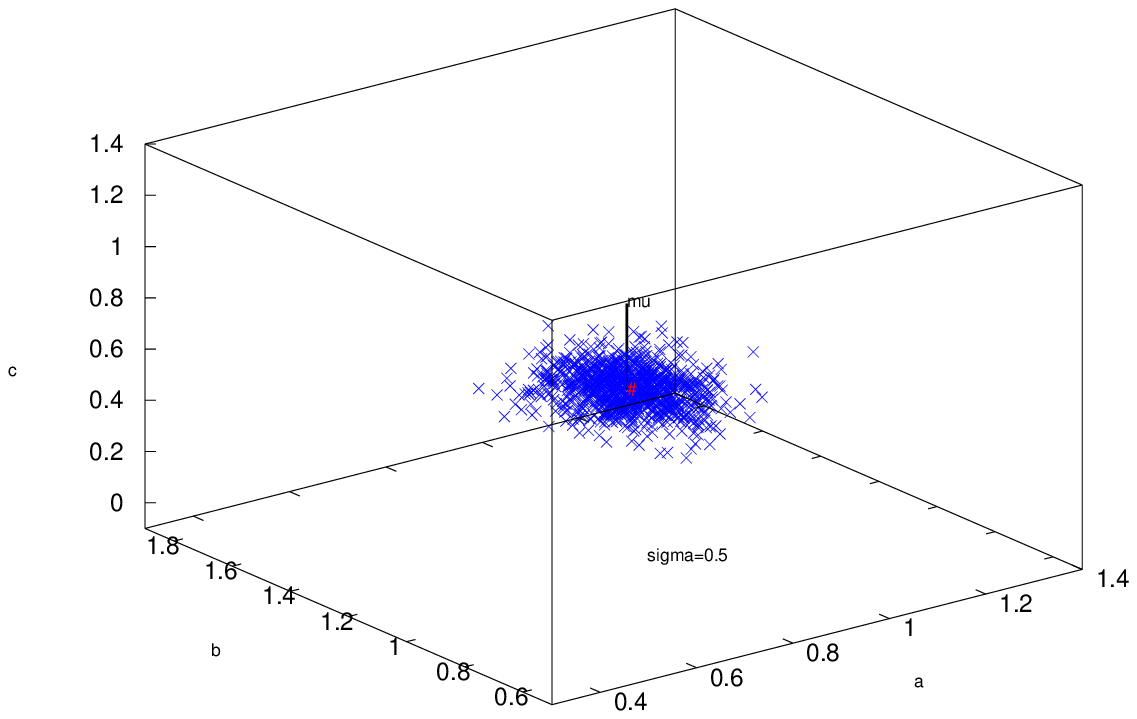} &
\includegraphics[width=4cm, height=4cm]{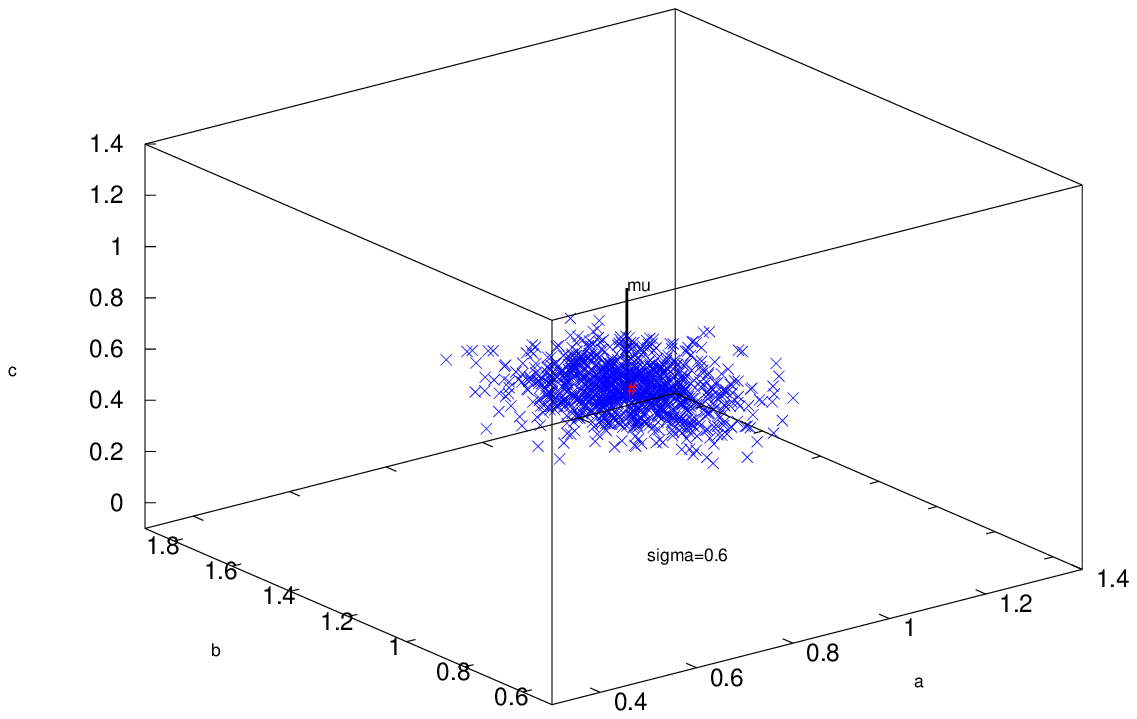} \\
\includegraphics[width=4cm, height=4cm]{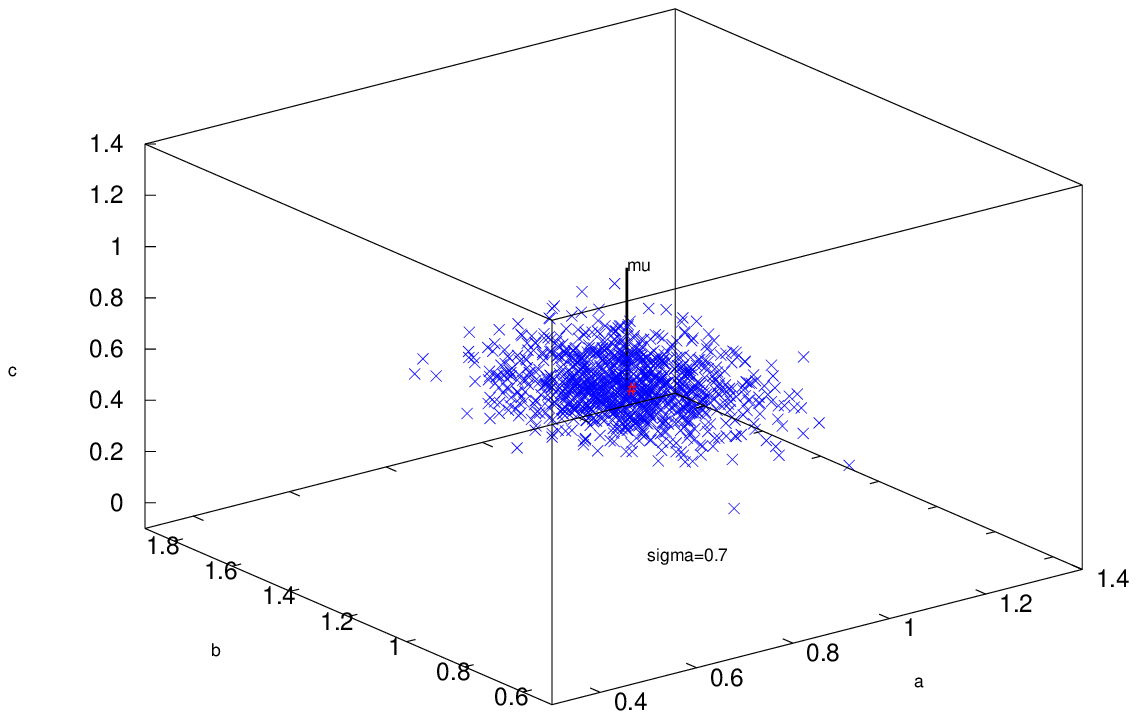} &
\includegraphics[width=4cm, height=4cm]{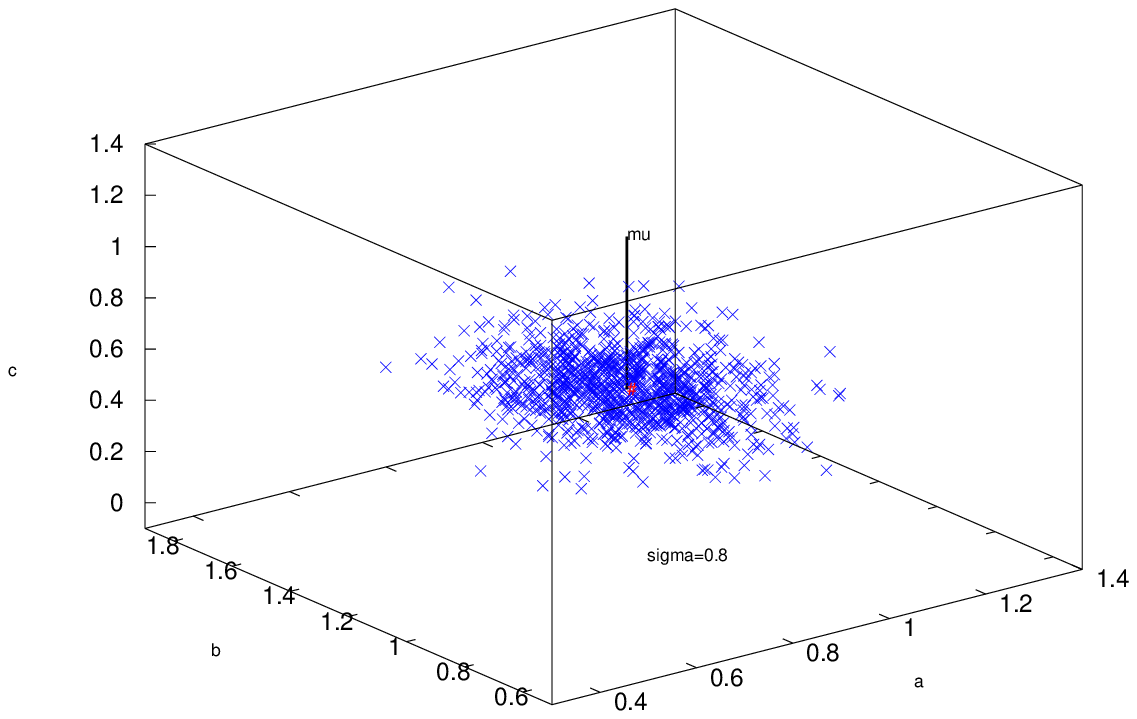} &
\includegraphics[width=4cm, height=4cm]{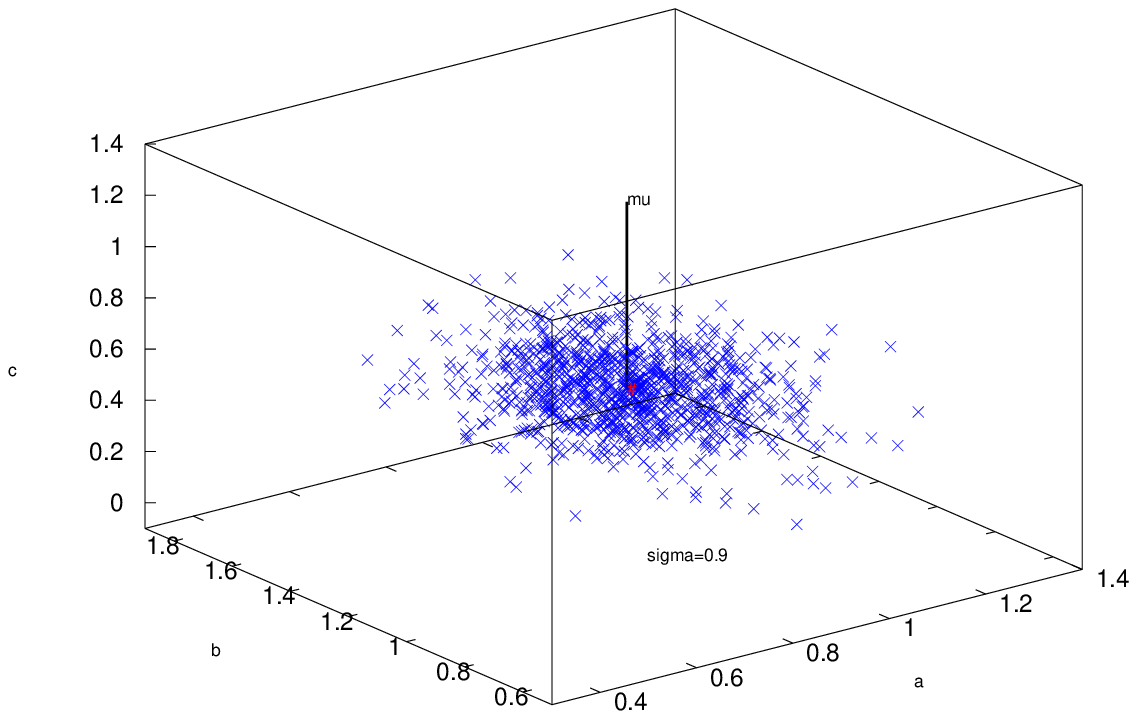}
\end{tabular}
\caption{The empirical distribution of the regressor $\hat \gamma =
  \hat \gamma(\y)$. See text for further information.} \label{fig:reg-id-s2-cloud}
\end{figure}
}
\end{center}
\begin{center}
{
\psfrag{estimator}{}
\psfrag{sigma=0.1}{{\tiny $\sigma=0.1$}}
\psfrag{sigma=0.2}{{\tiny $\sigma=0.2$}}
\psfrag{sigma=0.3}{{\tiny $\sigma=0.3$}}
\psfrag{sigma=0.4}{{\tiny $\sigma=0.4$}}
\psfrag{sigma=0.5}{{\tiny $\sigma=0.5$}}
\psfrag{sigma=0.6}{{\tiny $\sigma=0.6$}}
\psfrag{sigma=0.7}{{\tiny $\sigma=0.7$}}
\psfrag{sigma=0.8}{{\tiny $\sigma=0.8$}}
\psfrag{sigma=0.9}{{\tiny $\sigma=0.9$}}
\psfrag{sigma=1}{{\tiny $\sigma=1$}}
\psfrag{ks=0.76}{{\tiny $p=0.76$}}

\psfrag{ks=0.47}{{\tiny $p=0.47$}}

\psfrag{ks=0.81}{{\tiny $p=0.81$}}

\psfrag{ks=0.94}{{\tiny $p=0.94$}}

\psfrag{ks=0.80}{{\tiny $p=0.80$}}

\psfrag{ks=0.92}{{\tiny $p=0.92$}}

\psfrag{ks=0.71}{{\tiny $p=0.71$}}

\psfrag{ks=0.76}{{\tiny $p=0.76$}}

\psfrag{ks=0.92}{{\tiny $p=0.92$}}

\psfrag{ks=0.84}{{\tiny $p=0.84$}}

\psfrag{ks=0.54}{{\tiny $p=0.54$}}

\psfrag{ks=0.99}{{\tiny $p=0.99$}}

\psfrag{ks=0.37}{{\tiny $p=0.37$}}

\psfrag{ks=0.58}{{\tiny $p=0.58$}}

\psfrag{ks=0.94}{{\tiny $p=0.94$}}

\psfrag{ks=0.83}{{\tiny $p=0.83$}}

\psfrag{ks=0.88}{{\tiny $p=0.88$}}

\psfrag{ks=0.43}{{\tiny $p=0.43$}}

\psfrag{ks=0.65}{{\tiny $p=0.65$}}

\psfrag{ks=0.98}{{\tiny $p=0.98$}}

\psfrag{ks=0.68}{{\tiny $p=0.68$}}

\psfrag{ks=1.00}{{\tiny $p=1.00$}}

\psfrag{ks=0.51}{{\tiny $p=0.51$}}

\psfrag{ks=0.65}{{\tiny $p=0.65$}}

\psfrag{ks=0.80}{{\tiny $p=0.80$}}

\psfrag{ks=0.84}{{\tiny $p=0.84$}}

\psfrag{ks=0.73}{{\tiny $p=0.73$}}

\begin{figure}[htb]
\begin{tabular}{c|c|c}
\includegraphics[width=4cm, height=4cm]{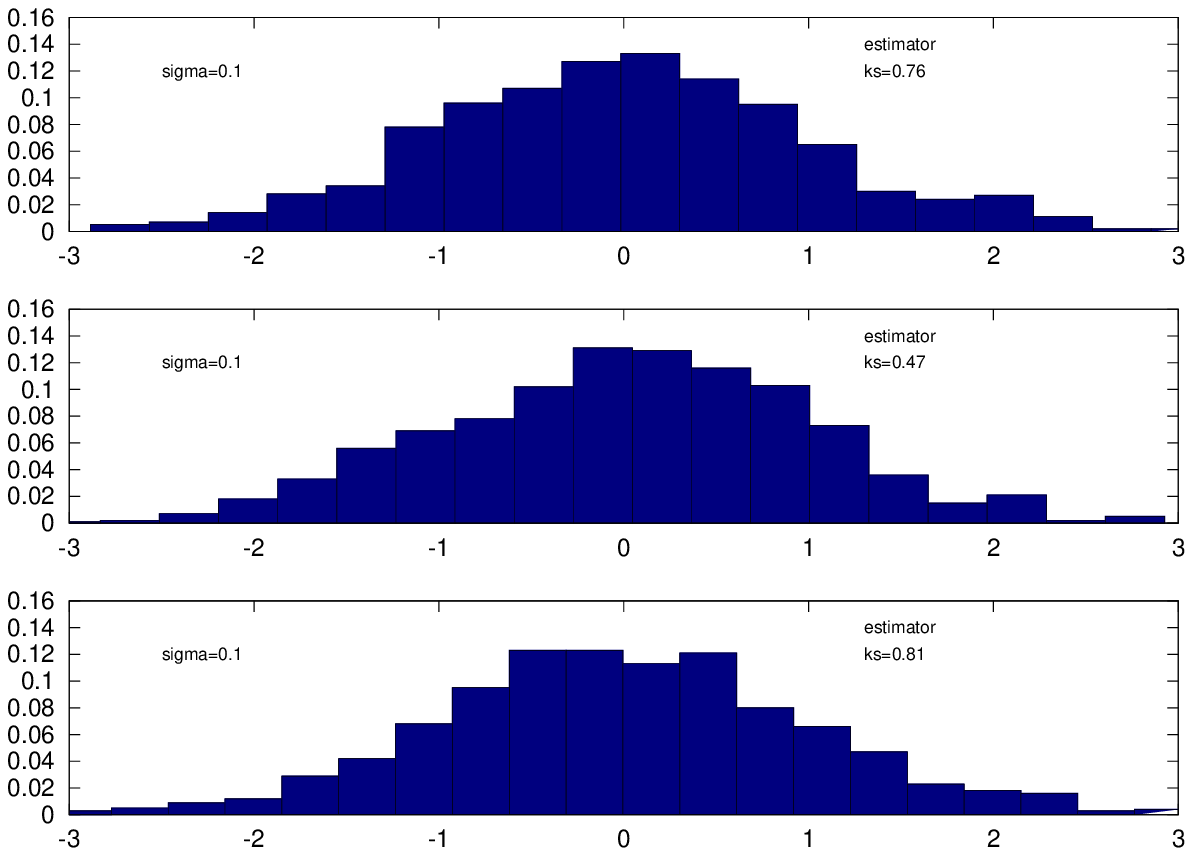} &
\includegraphics[width=4cm, height=4cm]{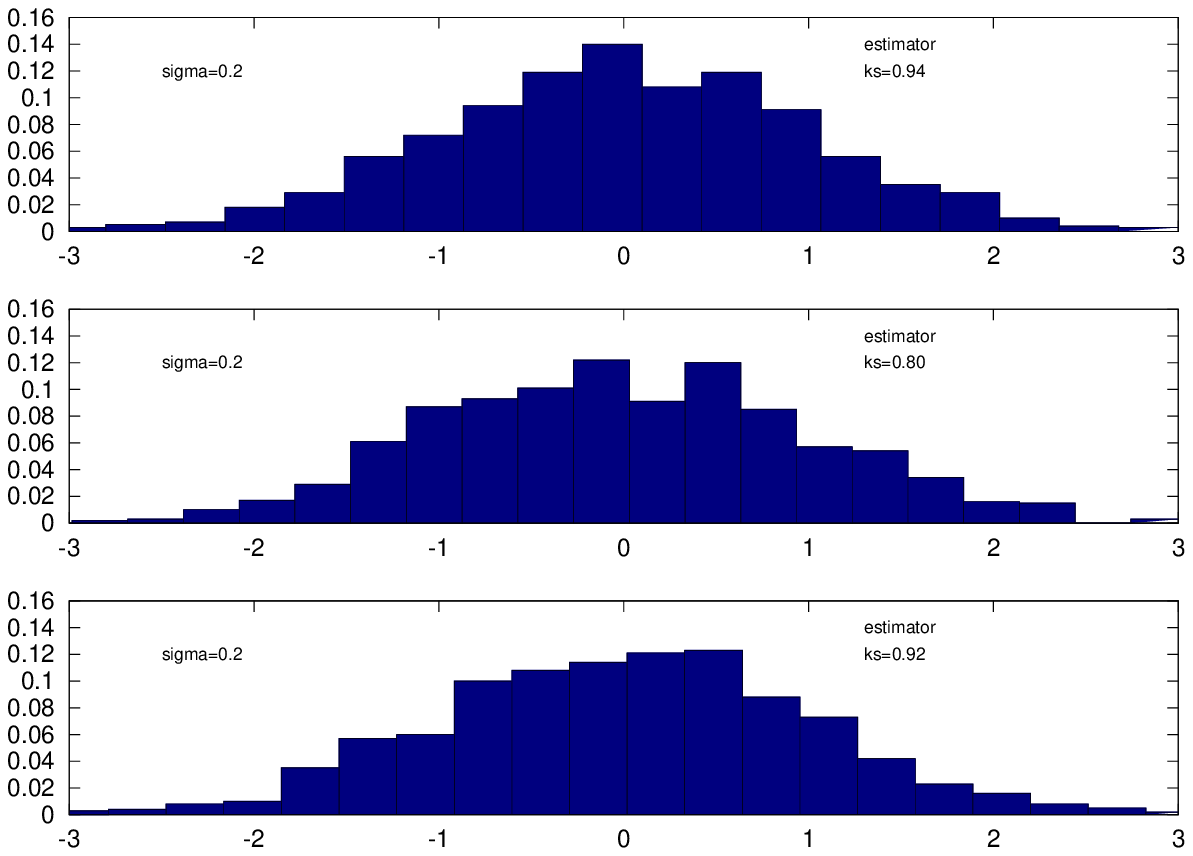} &
\includegraphics[width=4cm, height=4cm]{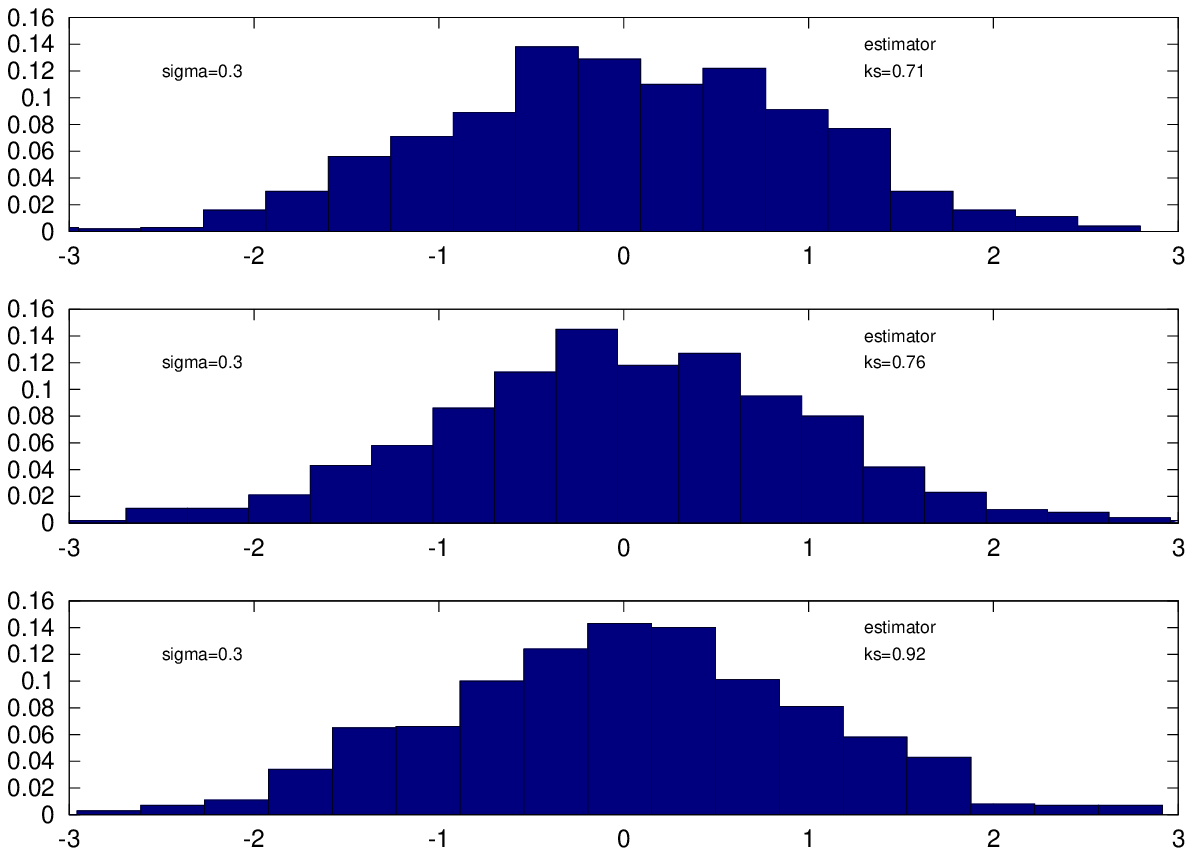} \\
\hline
\includegraphics[width=4cm, height=4cm]{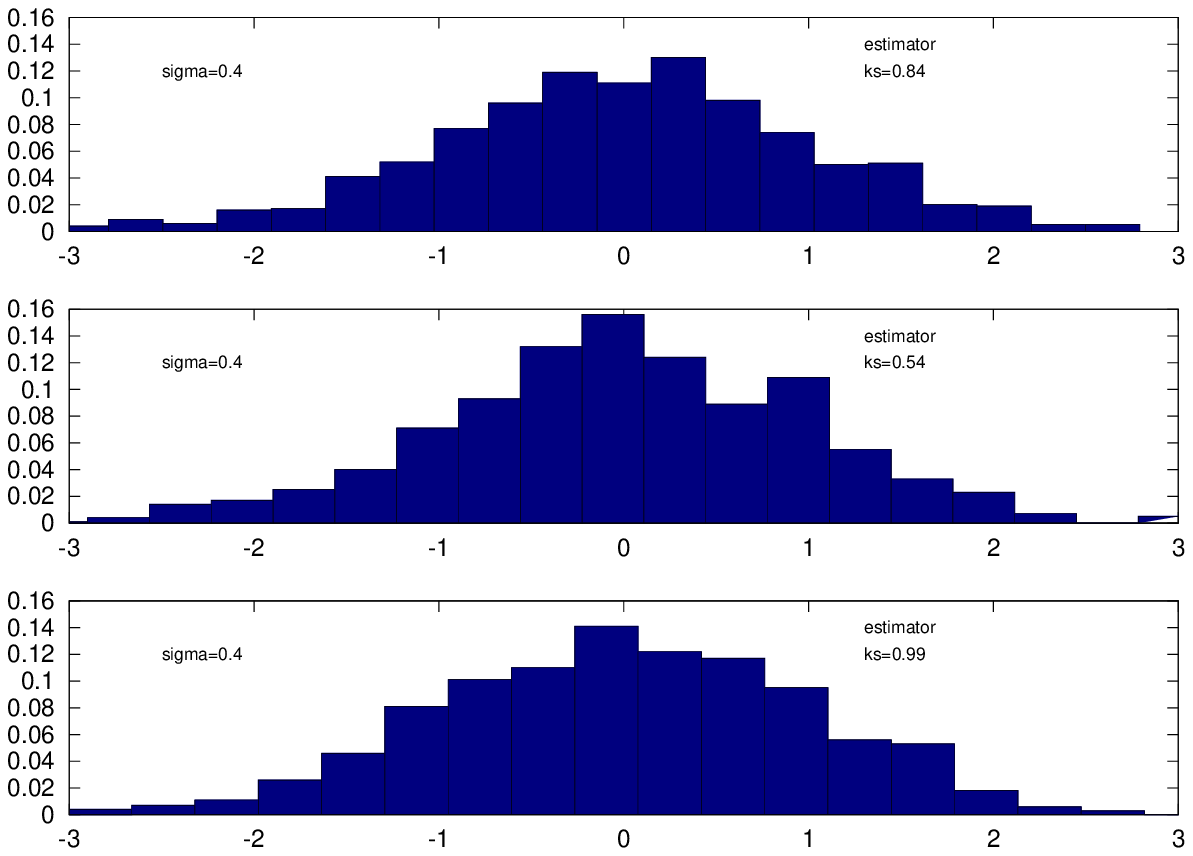} &
\includegraphics[width=4cm, height=4cm]{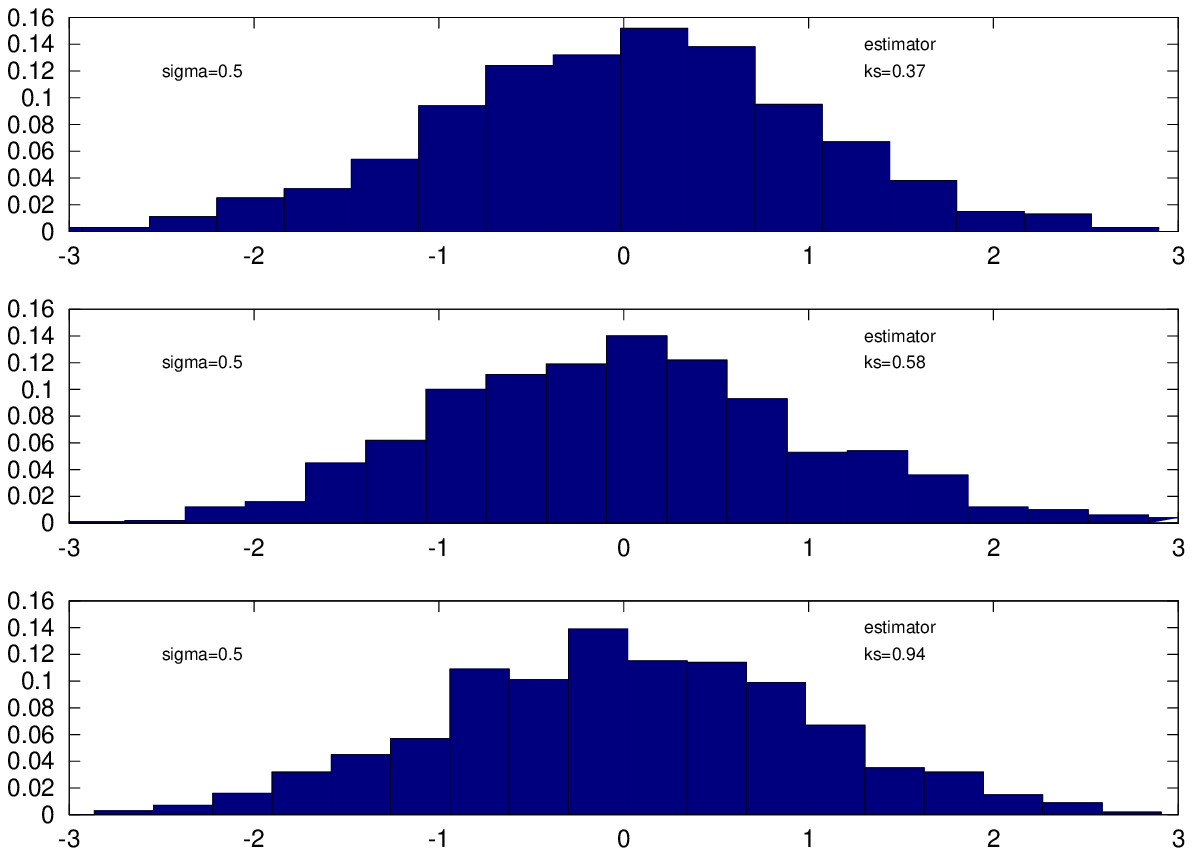} &
\includegraphics[width=4cm, height=4cm]{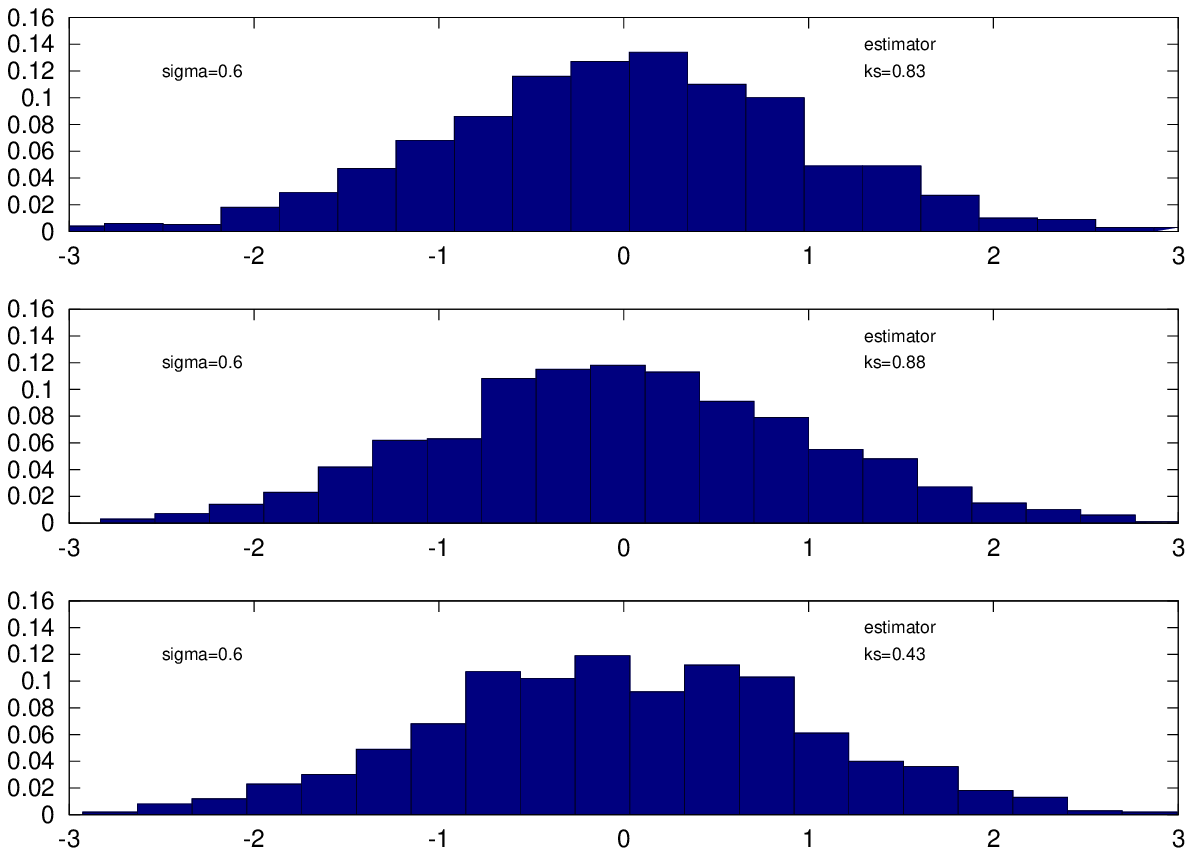} \\
\hline
\includegraphics[width=4cm, height=4cm]{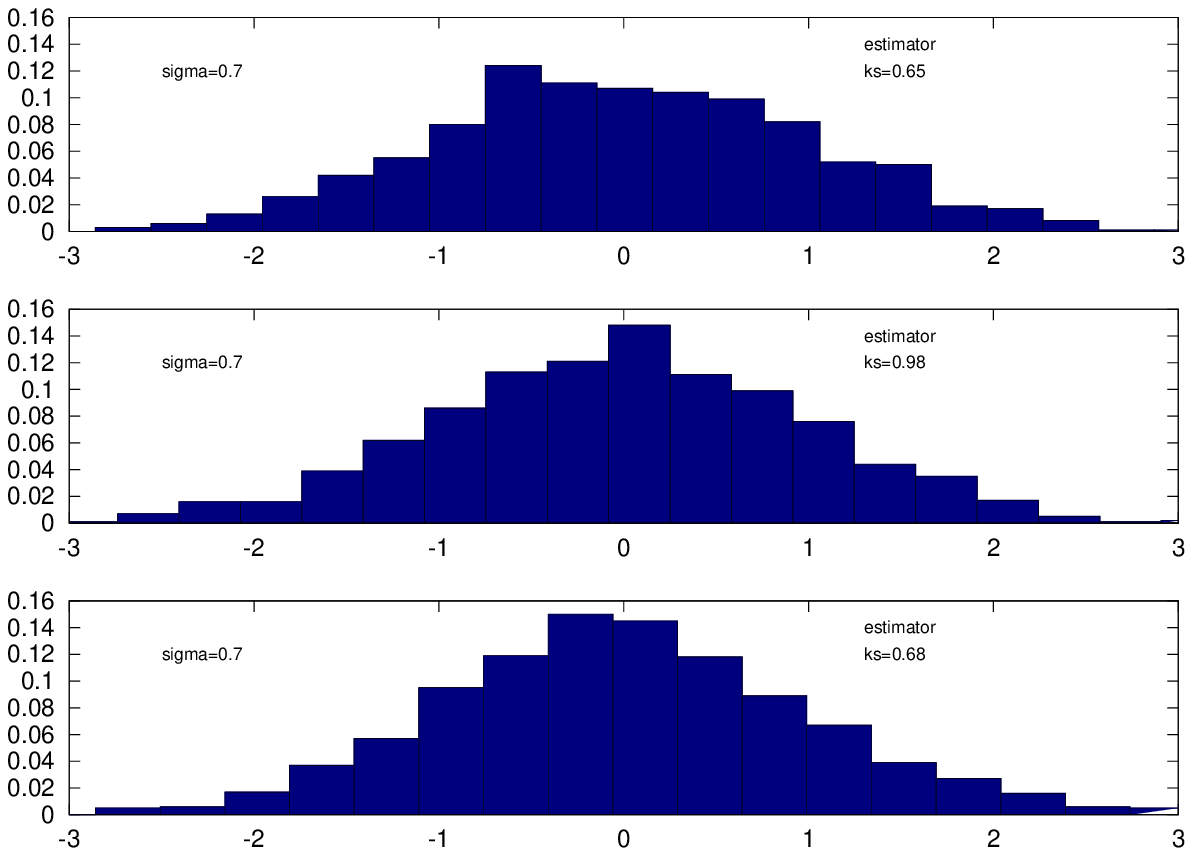} &
\includegraphics[width=4cm, height=4cm]{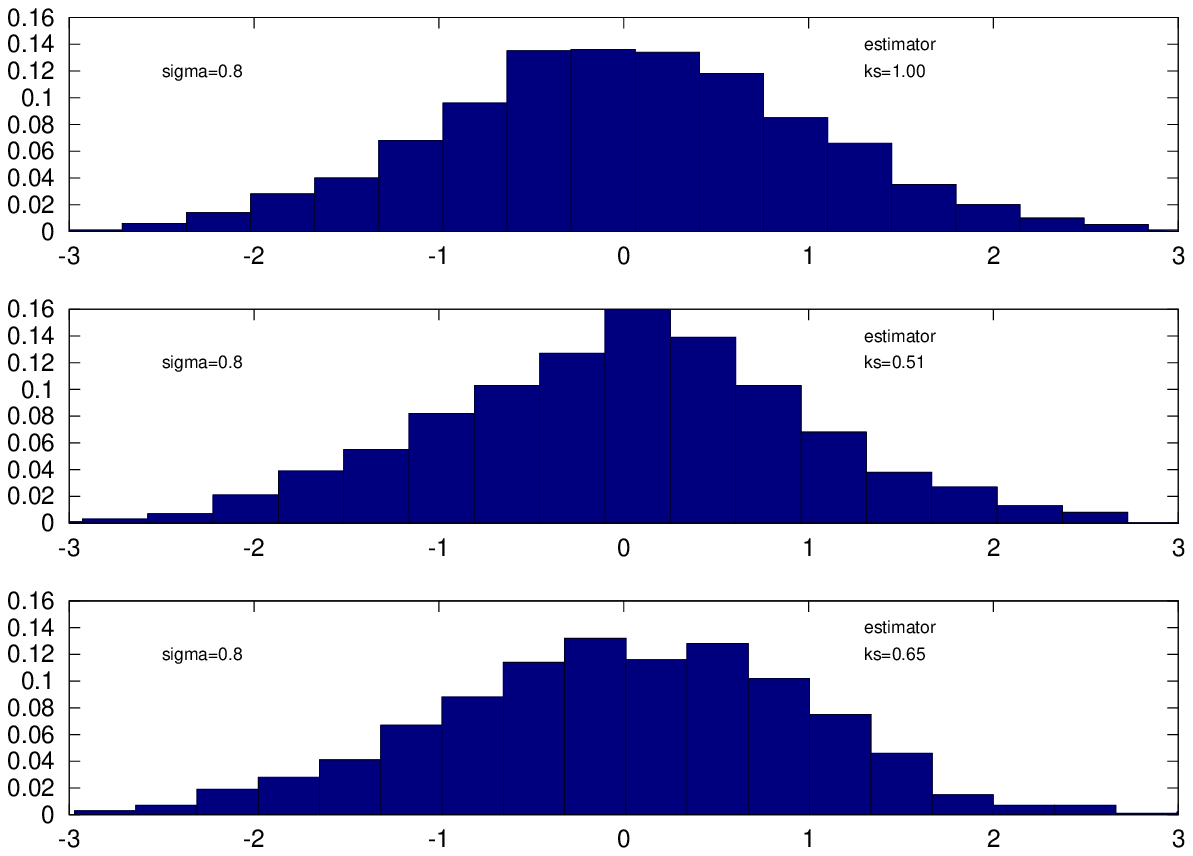} &
\includegraphics[width=4cm, height=4cm]{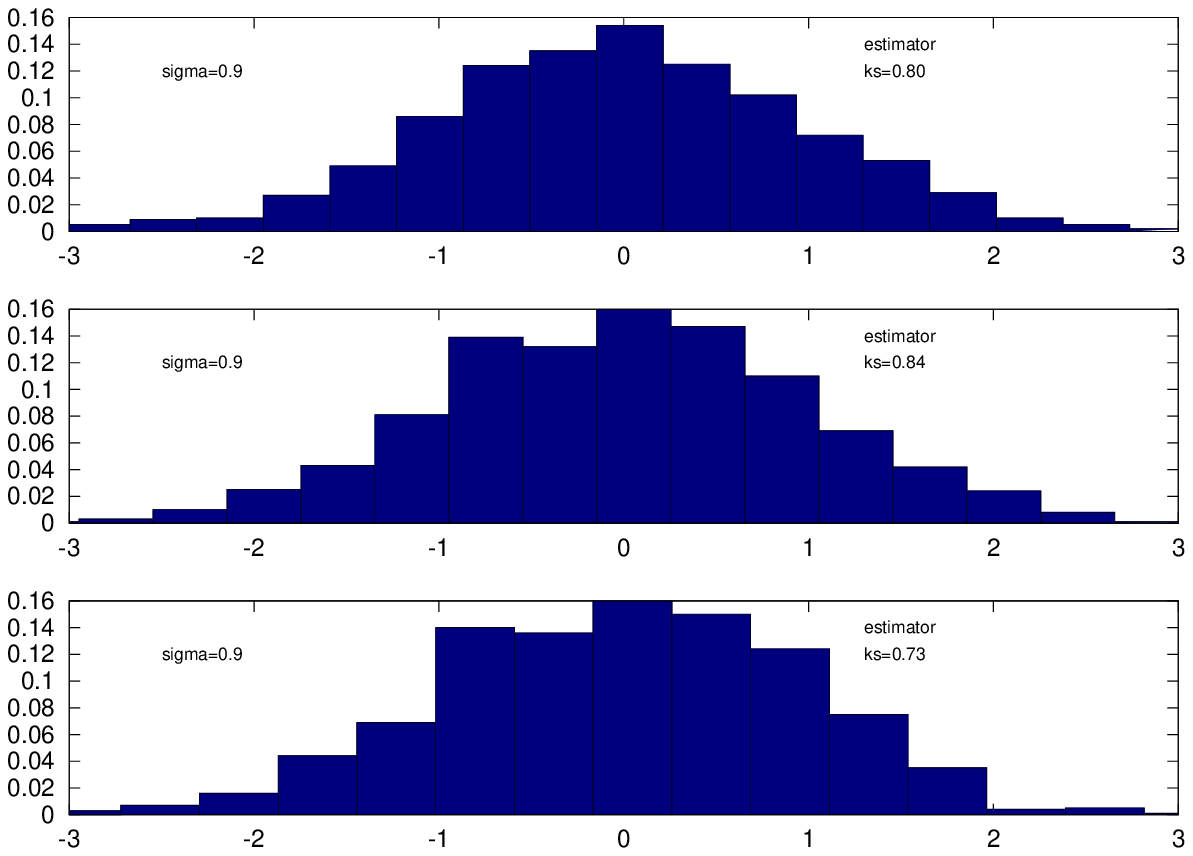}
\end{tabular}
\caption{The histogram of the normalised deviations from the mean of
  the regressor $\hat \gamma =
  \hat \gamma(\y)$. The p-value for the Kolmogorov-Smirnov test of
  normality is reported.} \label{fig:reg-id-s2-distribution}
\end{figure}
}
\end{center}
}

}
\end{example}

\subsection{A bayesian approach} \label{ssec:reg-bayes} 
Let $\y = (y_1,\ldots,y_k) \in \Phi = \Lambda^k$ and let
$\lambda(\gamma)\, \d \gamma$ be a bayesian prior on $\Gamma$
($\Gamma$ is only assumed to be a smooth submanifold of
$C^{\infty}(\Theta,\Lambda)$ at this point). Let $\ell : \Gamma \times
\Gamma \to \R$ be a loss function as defined in the introduction to
section \ref{sec:reg} and assume that $f(\y|\gamma) = \prod_i
f(y_i|\gamma(\theta_i))$ is the conditional density of $\y$. The
bayesian risk of $\hat\gamma \in \Gamma$ is then
\begin{align} \label{al:bay-risk}
\brisk{\hat\gamma} &= \int_{\y \in \Phi} \int_{\gamma\in \Gamma}
\ell(\hat \gamma, \gamma)\, f(\y|\gamma)\, \lambda(\gamma)\, \d\y\, \d \gamma.
\end{align}
One can define quantities
\begin{align} \label{al:bay-k}
\mu(\y) &= \int_{\gamma\in \Gamma} f(\y|\gamma)\, \lambda(\gamma)\, \d
\gamma, & \lambda(\gamma|\y) = \frac{f(\y|\gamma)\,
  \lambda(\gamma)}{\mu(\y)}, \notag\\
\brisk{\hat\gamma|\y}&= \int_{\gamma\in \Gamma} \ell(\hat
  \gamma,\gamma)\, \lambda(\gamma|\y)\, \d \gamma
\end{align}
to arrive at
\begin{align} \label{al:bay-cov}
\brisk{\hat\gamma} &= \int_{\y \in \Phi} \brisk{\hat \gamma|\y}\, \mu(\y)\,\d\y,
\end{align}
where dependence on the design points $\theta_i$ has been omitted for
notational compactness. Therefore, one can choose a bayesian
estimator $\hat \gamma$ by minimising the posterior risk
\begin{align} \label{al:bay-est}
g(\y) &= \argmin \left\{ \brisk{\hat\gamma | \y} \ : \ \hat\gamma
\in \Gamma  \right\}, && g : \Phi \to \Gamma.
\end{align}
Since $\Gamma$ is assumed to be compact, the bayesian estimator $g$ is
defined for all $\y$ and measurable. If, in addition, $\ell$ is
smooth, then $\brisk{\hat\gamma|\y}$ is smooth in $\hat\gamma$.

The following notation is useful in formulating the first-order
necessary condition to determine $g(\y)$. Let $\ell=\ell(\hat
\gamma,\gamma)$ be a smooth function that is defined for all pairs of
maps $\hat \gamma, \gamma$ in $C^{\infty}(\Theta,\Lambda)$. One may
view $\ell$ as a function of $\hat \gamma$ depending on the parameter
$\gamma$. Let
\begin{align} \label{al:bay-pd}
\frac{\partial \ell}{\partial \hat \gamma} &\in T^*_{\hat \gamma}
C^{\infty}(\Theta,\Lambda) 
\end{align}
be the $1$-form defined by fixing $\gamma$ and taking the derivative
with respect to $\hat \gamma$. In this case, the map $\gamma \mapsto
\frac{\partial \ell}{\partial \hat \gamma}$ is a smooth map from
$C^{\infty}(\Theta,\Lambda)$ to the vector space $T^*_{\hat \gamma}
C^{\infty}(\Theta,\Lambda)$.

\begin{proposition} \label{pr:bay-foc}
Assume that the loss function $\ell$ is a smooth function on
$C^{\infty}(\Theta,\Lambda) \times C^{\infty}(\Theta,\Lambda)$. Then, 
\begin{align} \label{al:bay-foc-ddr}
\dfrac{\partial \phantom{\hat \gamma}}{\partial \hat \gamma}
\brisk{\hat \gamma|\y} &= \int_{\gamma\in \Gamma} \frac{\partial
  \ell}{\partial \hat \gamma}\, \lambda(\gamma|\y)\, \d \gamma
\end{align}
If $\hat \gamma = g(\y)$ is a bayesian estimator satisfying
\eqref{al:bay-est}, then
\begin{align} \label{al:bay-foc}
\dfrac{\partial \phantom{\hat \gamma}}{\partial \hat \gamma}
\brisk{\hat \gamma|\y} & \text{ lies in } \nb[\hat \gamma]{\Gamma} \subset T^*_{\hat \gamma}
C^{\infty}(\Theta,\Lambda).
\end{align}
\end{proposition}

The proof of this proposition is straightforward. One observes that
the integral on the right-hand side is well defined since, by
\eqref{al:bay-pd}, one is integrating a smooth function which
takes values in a single vector space.

\subsubsection{The squared-norm loss function} \label{sssec:bay-sqn}
As $(\Lambda,\metric{h})$ is assumed to be isometrically embedded in
$(\E,\sigma)$ as in \eqref{eq:cd-1}, one may define an $L^2$ metric on
$C^{\infty}(\Theta,\Lambda)$ by means of the ambient euclidean structure
\begin{align} \label{al:bay-sqn}
|\gamma|^2 &= \int_{\theta\in \Theta} |\gamma(\theta)|^2\, \d \theta
&& \forall \gamma \in C^{\infty}(\Theta,\Lambda).
\end{align}
A natural squared-norm loss function is then
\begin{align} \label{al:bay-loss}
\ell(\hat \gamma,\gamma) &= |\hat \gamma-\gamma|^2 && \forall\hat
\gamma,\gamma \in C^{\infty}(\Theta,\Lambda).
\end{align}
(The requisite '$\j$'s in (\ref{al:bay-sqn}--\ref{al:bay-loss}) are
suppressed for simplicity). 

\begin{proposition} \label{pr:bay}
Let $\Gamma \subset C^{\infty}(\Theta,\Lambda)$ be a smooth
submanifold and the loss function $\ell$ be defined as in
\eqref{al:bay-loss}. If $\hat{\gamma}=g(\y)$ is a bayesian estimator
as in \eqref{al:bay-est}, then
\begin{align} \label{al:bay-est-on}
\j\bar{\gamma} &:= \int_{\gamma\in \Gamma} \j\gamma\, \lambda(\gamma|\y)\,
\d \gamma && \text{satisfies } \j\bar{\gamma} \in
\nb[\hat{\gamma}]{\Gamma} \subset T^*_{\hat \gamma}C^{\infty}(\Theta,\E).
\end{align}
\end{proposition}

\begin{remark} \label{re:bay}
{\rm

One considers $\j\bar{\gamma}$ to be a form in $T^*_{\hat \gamma}
C^{\infty}(\Theta,\E)$ and not in $T^*_{\hat \gamma}
C^{\infty}(\Theta,\Lambda)$ in equation \eqref{al:bay-est-on} due to
the natural embedding $\Lambda \subset \E$.

}
\end{remark}

\begin{proof}
In this case, the smoothness of $\brisk{\hat\gamma|\y}$ in $\hat\gamma$ is
immediate from the loss function. One computes that 
\begin{align}
\frac{\partial \ell}{\partial \hat \gamma} &= \int_{\theta\in \Theta}
\d_{\hat \gamma(\theta)} \j^*\, \left( \j\hat \gamma(\theta)-\j
\gamma(\theta)  \right)\, \d \theta && \in T_{\hat \gamma}^*
C^{\infty}(\Theta,\E),\notag\\
\intertext{whence}
\frac{\partial \phantom{\hat \gamma}}{\partial \hat \gamma} \brisk{\hat \gamma|\y} &=
\int_{\theta\in \Theta} \d \theta \, \d_{\hat \gamma(\theta)} \j^*
\left\{ \int_{\gamma\in \Gamma} \d \gamma \, \lambda(\gamma|\y)\,
\left( \j\hat \gamma(\theta)-\j
\gamma(\theta) \right) \right\}. \label{al:bay-dell}
\end{align}
Proposition \ref{pr:bay-foc} shows that the left-hand side of
\eqref{al:bay-dell} lies in $\nb[\hat \gamma]{\Gamma}$ if $\hat
\gamma$ is a bayesian estimator. Define $\xi \in T^*_{\hat \gamma}
C^{\infty}(\Theta,\E)$ by
\begin{align}
\xi(\theta) &= \d_{\hat \gamma(\theta)} \j^*
\left\{ \int_{\gamma\in \Gamma} \d \gamma \, \lambda(\gamma|\y)\,
\left( \j\hat \gamma(\theta)-\j
\gamma(\theta) \right) \right\} && \forall \theta\in \Theta.  \label{al:bay-xi}
\intertext{One observes that the right-hand side of
\eqref{al:bay-dell} vanishes on $T_{\hat \gamma}\Gamma$ if $\xi$
vanishes on $T_{\hat \gamma}\Gamma$, and $\xi$ vanishes on $T_{\hat
  \gamma}\Gamma$ if $\xi$ vanishes, \ie\hspace{-1mm}, if}
\j \hat \gamma(\theta) &\equiv \int_{\gamma\in \Gamma} \d \gamma \,
\lambda(\gamma|\y)\, \j
\gamma(\theta) & \bmod \nb[\hat \gamma]{\Gamma}_{\theta},\notag\\
\intertext{where $\nb[\hat \gamma]{\Gamma}_{\theta}$ is the subspace
  of $T^*_{\hat \gamma(\theta)} \E$ generated by elements of $\nb[\hat
  \gamma]{\Gamma}$ (which are sections of $\hat \gamma^*T^* \E$)
  evaluated at $\theta$. Therefore, one obtains}
\j \hat \gamma &\equiv \int_{\gamma\in \Gamma} \d \gamma \,
\lambda(\gamma|\y)\,
\j \gamma & \bmod \nb[\hat \gamma]{\Gamma}, \label{al:bay-xi-0}
\end{align}
which proves the proposition.
\end{proof}

\subsubsection{Estimation of Linear Maps} \label{sssec:lin}

Assume that both $\Theta$ and $\Lambda$ are isometrically embedded in
euclidean spaces $\E_0$ and $\E_1$ respectively. Let $\Gamma \subset
\Hom{\E_0,\E_1}$ be a manifold of linear maps that maps $\Theta$ to
$\Lambda$. Inspection of the right-hand side of \eqref{al:bay-est-on}
shows that $\bar \gamma$ is itself the restriction of a linear map to
$\Theta$, so the bayesian estimator $\hat \gamma$ can be described
using only the geometry of $\Hom{\E_0,\E_1}$. 

Define a positive semi-definite quadratic form on $\Hom{\E_0,\E_1}$ by
\begin{align}
\langle\langle \alpha  , \beta \rangle\rangle &= \trace{\alpha'\cdot
\beta \cdot \tau} && \forall \alpha,\beta \in \Hom{\E_0,\E_1}, \notag\\
\intertext{where}
\tau &= \int_{\theta \in \Theta} \j(\theta) \otimes \j(\theta)'\, \d
  \theta \in \Hom{\E_0,\E_0}.  \label{al:ip01}
\end{align}
The first-order condition \eqref{al:bay-est-on} implies that the
bayesian estimator $\hat \gamma$ satisfies
\begin{align} \label{al:lin-foc}
\hat \gamma \cdot \tau &\equiv \bar \gamma \cdot \tau & \bmod
\nb[\hat \gamma]{\Gamma},
\end{align}
where $\nb[\hat \gamma]{\Gamma}$ is the normal space to $T_{\hat
\gamma}\Gamma$ in $\Hom{\E_0,\E_1}$.

\begin{proposition} \label{pr:bay-lin}
Let $\Gamma \subset \Hom{\E_0,\E_1}$ be a submanifold and the loss
function $\ell$ be defined as in \eqref{al:bay-sqn}. Suppose that
$\Theta$ spans $\E_0$. If $\hat{\gamma}=g(\y)$ is a bayesian estimator
as in \eqref{al:bay-est}, then the linear transformation
\begin{align} \label{al:bay-est-on-lin}
\bar{\gamma} &:= \int_{\gamma\in \Gamma} \gamma\, \lambda(\gamma|\y)\,
\d \gamma && \text{satisfies }  \bar{\gamma} \equiv
\hat{\gamma} \bmod \nb[\hat{\gamma}]{\Gamma} \cdot \tau^{-1}.
\end{align}
\end{proposition}

\begin{proof}
The only thing that remains to prove is that $\tau$ is non-degenerate
if $\Theta$ spans $\E_0$. If $v \in \Hom{\E_0,\E_1}$ and
\begin{align} \label{al:ka}
0 = \left\langle\langle v , v \right\rangle\rangle &=
 \int_{\theta \in \Theta} |v\cdot \theta|^2\, \d \theta, &\text{then }&
 \Theta \subseteq \ker v.
\end{align}
Therefore, $\E_0 = \spn{\Theta} \subseteq \ker v,$ so $v=0$.
\end{proof}

\begin{remark} \label{re:kim}
{\rm

Let $\Theta$ be the unit sphere in $\E_0$. One computes that
$\tau$ is a scalar multiple of the identity matrix, whence condition
\eqref{al:bay-est-on-lin} is simply that $\hat \gamma$ is the
orthogonal projection onto $\Gamma$ of $\bar \gamma$.

}
\end{remark}

Let $\ell$ be the loss function on $\Gamma$ induced by the
inner product on $\Hom{\E_0,\E_1}$:
\begin{align}
\ell(\hat \gamma, \gamma) &= |\hat{\gamma}-\gamma|^2 = \trace{
  (\hat{\gamma}-\gamma)'(\hat{\gamma}-\gamma) }
 && \forall \hat \gamma, \gamma \in \Gamma.  \label{al:bay-sqn-l}
\end{align}
When $\E_0=\E_1$ and $\Gamma \subset \Orth{\E}$, the loss function
simplifies to $2s - 2\, \trace{\hat{\gamma}' \gamma}$ where
$s=\dim\E$.

\subsubsection{The intrinsic distance loss function} \label{sssec:intrinsic-d}

Let $\hat \gamma, \gamma \in C^{\infty}(\Theta,\Lambda)$ be smooth
maps between the riemannian manifolds $(\Theta,\metric{g})$ and
$(\Lambda,\metric{h})$. For each $\theta \in \Theta$, let $w(\theta)
\in T_{\hat \gamma(\theta)} \Lambda$ be a tangent vector to a shortest
geodesic $c(s) = \exp_{\hat \gamma(\theta)}(s\cdot w(\theta))$ such
that $c(1) = \gamma(\theta)$.

If $\gamma(\theta)$ does not lie in the cut locus of $\hat
\gamma(\theta)$, the tangent vector $w(\theta)$ is uniquely defined
and one may unambiguously write $w(\theta) = \log_{\hat
\gamma(\theta)}(\gamma(\theta))$. It is apparent that there are
measurable maps $\theta \mapsto w(\theta)$, and this map is smooth off
the above-mentioned set of ``bad'' points. In particular, if the graph
of $\gamma$ lies in a tubular neighbourhood of the graph of $\hat
\gamma$, then the map $w$ is a uniquely defined, smooth map.


Let $\cutlocus=\cutlocus_{\gamma,\hat \gamma} \subset \Theta$ be the
set of points $\theta$ such that $\gamma(\theta)$ lies in the cut
locus of $\hat \gamma(\theta)$. If the measure of $\cutlocus$ is zero,
then compactness of $\Lambda$ implies that $w$ is
square-integrable. Therefore, one may define a one-form $\omega =
\omega_{\gamma,\hat \gamma} \in T_{\hat \gamma}^*
C^{\infty}(\Theta,\Lambda)$ by
\begin{equation} \label{eq:omega-id}
\left\langle \omega,v\right\rangle = \int_{\theta \in \Theta}\ \d
\theta \cdot \metric{h}( w(\theta) , v(\theta) )_{\hat \gamma(\theta)}.
\end{equation}
for each $v \in T_{\hat \gamma}
C^{\infty}(\Theta,\Lambda)$.

\begin{proposition} \label{pr:id}
Let 
\begin{equation} \label{eq:loss-id}
\ell(\hat \gamma,\gamma) = \frac{1}{2} \, \int_{\theta \in \Theta}\ \d
\theta \cdot \dist(\hat \gamma(\theta),\gamma(\theta))^2,
\end{equation}
where $\dist$ is the riemannian distance function of
$(\Lambda,\metric{h})$. If $\cutlocus_{\gamma,\hat \gamma}$ has
measure zero and $\Lambda$ is compact, then $\dfrac{\partial
\ell}{\partial \hat \gamma}$ exists at $(\gamma,\hat \gamma)$ and
equals
\begin{equation} \label{eq:loss-pd}
\frac{\partial \ell}{\partial \hat \gamma} = -\omega_{\gamma,\hat\gamma}.
\end{equation}
\end{proposition}

\begin{proof}
Let $\hat \gamma_t$ be a curve of smooth maps such that $\hat
\gamma_{t=0} = \hat \gamma$ and $v = \ds \left.\ddt{t}\right|_{t=0}
\hat \gamma_t$. Let $\theta \in \Theta - \cutlocus$ be fixed, and let
$c_t(s)$ be the minimal geodesic from $\gamma(\theta)$ to $\hat
\gamma_t(\theta)$. It is clear from figure \ref{fig:jacobi} that the
derivative of $\frac{1}{2}\,\dist(\hat
\gamma_t(\theta),\gamma(\theta))^2$ is $\left\langle c'(1) , v(\theta)
\right\rangle$ where $c=c_0$. From the above discussion, it is clear
that $c'(1) = -w(\theta)$.

If $\cutlocus_{\gamma,\hat \gamma}$ has zero measure, then the
discussion above shows that $\dfrac{\partial \ell}{\partial \hat
\gamma}$ exists at $(\gamma,\hat \gamma)$ and equals
$-\omega_{\gamma,\hat \gamma}$.
\end{proof}

Let $\gcutlocus \subset \Gamma$ be the set of maps $\hat \gamma$ such
that $\displaystyle \int_{\gamma \in \Gamma}
\int_{\cutlocus_{\gamma,\hat \gamma}} \d \theta\,\d \gamma = 0$. By
proposition \ref{pr:id}, if $\hat \gamma \in \gcutlocus$, then
$\dfrac{\partial \ell}{\partial \hat \gamma}$ exists for almost all
$(\gamma,\hat \gamma) \in \Gamma \times \{\hat \gamma\}$. The
following theorem is a consequence of Proposition \ref{pr:id} and
Fubini's theorem.

\begin{theorem} \label{thm:intrinsic}
\begin{enumerate}
\item If $\hat \gamma \in \gcutlocus$ and $\hat{\gamma}=g(\y)$ is a
  bayesian estimator as in Proposition \ref{pr:bay-foc}, then $\hat
  \gamma$ satisfies
\begin{align} \label{al:intrinsic-weak}
\int_{\gamma \in \Gamma} \d \gamma \, \lambda(\gamma|\y) \, \left\langle\log_{\hat
  \gamma(\theta)}(\gamma(\theta)), v(\theta) \right\rangle &= 0 &&
\forall v \in T_{\hat{\gamma}}{\Gamma}.
\end{align}

\item In particular, if $\hat \gamma \in \gcutlocus$ satisfies the equation
\begin{align} \label{al:intrinsic-strong}
\int_{\gamma \in \Gamma} \d \gamma \, \lambda(\gamma|\y) \, \log_{\hat
  \gamma(\theta)}(\gamma(\theta)) &= 0 && \bmod \nb[\hat{\gamma}]{\Gamma}_{\theta}
\end{align}
for a.a. $\theta \in \Theta$, then $\hat \gamma$ is a candidate for a
bayesian estimator as in Proposition \ref{pr:bay-foc}.

\end{enumerate}
\end{theorem}

\begin{example} \label{ex:bayes-s2-intrinsic}
{

\def\x{{\bf x}}
\def\e{{\bf e}}
\rm

Let us examine an application of both parts of Theorem
\ref{thm:intrinsic}. Let $\Theta=\Lambda=S^2 \subset \E^3$ be the unit
sphere and let $\Gamma=\SO{3}$ be the group of orientation-preserving
isometries of $S^2$ with normalised Haar measure $\d \gamma$. 

\begin{enumerate}
\item Since $\Gamma$ is a transitive group of isometries, the
  logarithm function is $\Gamma$-equivariant, so part (1) of
  \ref{thm:intrinsic} implies that
  \begin{equation} \label{eq:bayes-s2-intrinsic-vtheta}
    w_{\hat \gamma}(\theta|\y) = \int_{\gamma \in \Gamma} \d \gamma\,
    \lambda(\hat \gamma \gamma |\y)\, \log_{\theta}(\gamma \theta)
  \end{equation}
  must integrate to zero on $S^2$ against any vector field of the form
  $v(\theta) = \xi \cdot \theta$, $\xi \in \so{3}$. One uses the fact
  that $\log_{\theta}(\gamma \theta) = \frac{\alpha}{\sin \alpha}
  \times (\gamma \theta - \left\langle \gamma \theta , \theta
  \right\rangle \cdot \theta)$ (\cf \ref{al:reg-id-s2-log}) to compute
  that
\begin{align}
\left\langle w_{\hat \gamma} , v \right\rangle &= \int_{\theta \in
  S^2} \int_{\gamma \in \SO{3}}\ \d \theta\, \d \gamma\,
\lambda(\hat{\gamma} \gamma|\y)\, \frac{\alpha}{\sin \alpha}\, \left\langle \gamma \theta , \xi
\theta \right\rangle, && \cos \alpha = \left\langle \gamma \theta ,
\theta \right\rangle\\
&= -\frac{1}{3} \times \trace{\tau(\hat{\gamma}) \cdot
  \xi}, \label{al:bayes-s2-intrinsic-vanish}
\intertext{where}
\tau(\hat{\gamma}) &= 3 \int_{\theta \in
  S^2} \int_{\gamma \in \SO{3}}\ \d \theta\, \d \gamma\,
\lambda(\hat{\gamma} \gamma |\y)\, \frac{\alpha}{\sin \alpha}\,
\gamma  \theta \theta'  && (\theta'=\textrm{transpose of }\theta)\label{al:bayes-s2-intrinsic-vanish-tau}
\end{align} is defined analogous to
\eqref{al:bay-est-on}. Since $\int_{\theta \in S^2} \d \theta\ \theta
\otimes \theta' = \frac{1}{3}I$, if the weight $\alpha/\sin \alpha$
were identically $1$, then $\bar{\gamma} = \hat{\gamma}
\tau(\hat{\gamma})$ would coincide with that defined in
\eqref{al:bay-est-on}. It follows that if $\hat{\gamma}$ equals the
bayesian estimator $g(\y)$, then $\tau(\hat{\gamma})$ must be
symmetric. In other words, $\hat{\gamma}$ is the orthogonal projection
of $\bar{\gamma}(\hat{\gamma}) = \hat{\gamma} \cdot
\tau(\hat{\gamma})$ onto $\Gamma$, similar to
\eqref{al:reg-id-s2-foc}.

\medskip

\item On the other hand, let us investigate condition (2) of Theorem
  \ref{thm:intrinsic}. Let $\x : (S^2)^k \to \SO{3}$ be an equivariant
  map and let the joint conditional density of $\y$ be
  $f(\y|\gamma) = 1 + c \trace{\gamma' \cdot \x(\y) }$. Assume that
  the mean of $\gamma'$ with respect to the bayesian prior
  $\lambda(\gamma)$ is zero. The posterior density
  $\lambda(\gamma|\y)$ is therefore equal to $f(\y|\gamma)$.

  To fix ideas, one may take $\x(\y) = \pi(\sum_{l=1}^k y_l \otimes
  \theta_l')$, where $\pi : \gl{3} \to \SO{3}$ is the orthogonal
  projection, and $\lambda(\gamma)=1$ for all $\gamma$.

  Let $\e \in S^2$ be a given point. Since $\Gamma$ acts transitively,
  one can write $\theta = \alpha \cdot \e$ for some $\alpha \in
  \Gamma$. One therefore finds that the vanishing of $w_{\hat
    \gamma}(\theta|\y)$ is equivalent to the vanishing of
  \begin{equation} \label{eq:vvtheta}
    \int_{\alpha \in \Gamma} \d \alpha \ \left| \int_{\gamma \in \Gamma} \d
    \gamma \, \lambda( \hat \gamma \alpha \gamma \alpha^{-1}|\y) \,
    \log_{\e} (\gamma\cdot\e) \right|^2.
  \end{equation}
  If one introduces Euler angles on $\SO{3}$ relative to an orthonormal
  frame $\{\e_1,\e_2,\e_3=\e\}$, then one can write $\gamma_j =
  R_3(a_j)R_1(b_j)R_3(c_j)$ where $R_i(s)$ is a rotation in the plane
  orthogonal to $\e_i$ counterclockwise by angle $s$. The vanishing of
  \eqref{eq:vvtheta} is equivalent to the vanishing of the multi-integral
  \begin{align}
    \int_{[0,2\pi]^4 \times [0,\pi]^2}\ & \frac{1}{64 \pi^4} \, \, \d a_1 \, \d a_2 \, \d c_1 \, \d c_2 \, \d
    b_1 \, \d b_2  \label{al:vtheta-is-0}\\
    & \times \sin(b_1)\sin(b_2) \, b_1 b_2 \cos(a_1-a_2) \,
    \lambda( \hat \gamma \alpha \gamma_1 \alpha^{-1}|\y) \,
    \lambda( \hat \gamma \alpha \gamma_2 \alpha^{-1}|\y)
  \end{align}
  for every $\alpha \in \SO{3}$.

  Let the special orthogonal matrix $\alpha^{-1} \hat \gamma' \x \alpha$
  be factorised as $R_3(x)R_1(y)R_3(z)$ in terms of Euler
  angles. \maxima computes the integral \eqref{al:vtheta-is-0} to be
  $\pi^4 c^2 \sin(y)^2/256$ \cite{maxima}. Therefore, the integral
  vanishes for all $\alpha$ iff $\hat \gamma' \x = I$, \ie ${\hat
    \gamma} = \x$.

\begin{maximacode}
Verification of \ref{al:vtheta-is-0}
scalarmatrixp : true;
logv(x,y) := y - (x.y)*x;
gama[i] := euler_matrix(a[i],b[i],c[i]);
thetal[i] := cast_mlf(gama[i].[0,0,1]);
q : expand((A.thetal) . (A.thetal));
I : mintegrate(q * sin(b),[[a,0,2*
if is(equal(I-tr_AtA/3,0)) then print("true") else print("false");
q : expand(thetal.(A.thetal));
q2 : expand(q^2);
J : expand((mat_trace(A)^2 + mat_trace(A^^2) + mat_trace(transpose(A).A))/15);
K : mintegrate(q2 * sin(b),[[a,0,2*
if is(equal(J-K,0)) then print("true") else print("false");

kill(gama,t,x,y,z,a,b,c,I,J);
bayes_lposterior(gama,x) := block([c], 1+c*mat_trace(transpose(gama).x));
bayes_eposterior(gama,x) := block([c], exp(c*mat_trace(transpose(gama).x)));
x0 : euler_matrix('x,'y,'z);
gama[i] := euler_matrix(-a[i],-b[i],-c[i]);
t[i] := bayes_lposterior(gama[i],x0);
I : sin(b[1])*sin(b[2])*b[1]*b[2]*cos(a[1]-a[2])*t[1]*t[2];
J : expand(mintegrate(I,[[a[1],0,2*
J : trigsimp(J);
if is(equal(J,
\end{maximacode}

\end{enumerate}

}
\end{example}

\end{document}